\documentclass[12pt,twoside,english,a4paper]{article}
\usepackage{babel,amssymb,amsthm}
\usepackage{graphicx}
\usepackage{amsmath}
\usepackage{bm}
\usepackage{float, framed}
\usepackage[utf8]{inputenc}
\usepackage{wrapfig}

\newcommand{\smallfrac}[2]{{\textstyle\frac{#1}{#2}}} 
\newcommand{\jump}[1]{[\![#1]\!]}
\newcommand{\ave}[1]{\{\!\!\{#1\}\!\!\}}
\newcommand{\triple}[1]{|\!|\!|#1|\!|\!|}

\newcommand{\bff}{\mathbf}

\newtheorem{proposition}{Proposition}[section]
\newtheorem{corollary}[proposition]{Corollary}
\newtheorem{lemma}[proposition]{Lemma}

\numberwithin{equation}{section}

\title{Boundary-Finite Element discretization \\ of time dependent 
acoustic scattering  \\ by elastic obstacles with piezoelectric behavior}

\date{\today}

\author{Tonatiuh S\'anchez-Vizuet
 \& Francisco--Javier Sayas\footnote{TSV and FJS partially funded by NSF grant DMS 1216356.}  \\
Department of Mathematical Sciences, University of Delaware, USA\\
{\tt \{tonatiuh,fjsayas\}@udel.edu }}

\begin{document}

\maketitle

\begin{abstract}
A coupled BEM/FEM formulation for the transient interaction between an acoustic field and a piezoelectric scatterer is proposed. The scattered part of the acoustic wave is represented in terms of retarded layer potentials while the elastic displacement and electric potential are treated variationally. This results in an integro-differential system. Well posedness of a general Galerkin semi-discretization in space of the problem is shown in the Laplace domain and translated into explicit stability bounds in the time domain. Trapezoidal-Rule and BDF2 Convolution Quadrature are used in combination with matching time stepping for time discretization. Second order convergence is proven for the BDF2-based method. Numerical experiments are provided for BDF2 and Trapezoidal Rule based time evolution.  
\end{abstract}

%
\section{Introduction}
%
The study of the interaction between acoustic waves and elastic structures has been subject of much work in recent decades. Many of the recent modeling and computational efforts have been driven by the need to develop and improve techniques for vibration control and reduction. Passive techniques rely on the use of sound absorbing materials that dissipate the energy of the acoustic wave and have been successfully used to damp high-frequency vibrations. On the other hand, active techniques employing piezoelectric materials exploit the adaptability of the piezoelectric solid to react to the vibrations in order to cancel them. Active materials are used to provide extra control in the low frequency range. 

In the frequency domain, works like \cite{DeLaOh:2008,DeLaOh:2009} have derived mathematical models and variational formulations suitable for numerical treatment of the process, their approach uses an effective load to model the action of the incident acoustic wave on the piezoelectric material and is geared towards a finite element solution of all the unknowns involved in the problem. In the time domain, \cite{AlHu:1970} is a classic reference for finite element simulation of waves in piezoelectrics, and a thorough review of the work done on this area up until the early 2000's can be found in \cite{Benjeddou:2000}. The propagation of plane waves waves in layered piezoelectric media has been addressed recently in \cite{FaLiPaWa:2008,YuZh:2012}, using analytical methods to study the reflection and transmission of plane waves at the interface of media with different material coefficients. Within the context of dynamic crack propagation in piezoelectric solids, finite element formulations have been explored in \cite{EnRiKu:2005,Nguyen:2012}, while Boundary Integral Equations (BIE) have been treated in \cite{Pan:1999,GaKo:2000,GaKoMoSc:2004}. Time domain BIE's for a purely piezoelectric problem have been used in recent works like \cite{GaZh:2008,QiZhZh:2015}, where a Nystr\"om  approach is followed for the space discretization and Convolution Quadrature is used in time. In both cases the model concerns only the propagation of the wave inside the piezoelectric material and no interaction with an acoustic wave is considered. 

The present article describes, discretizes, and analyzes the complete interaction problem, considering an incident acoustic wave that interacts with a piezoelectric scatterer through coupling boundary conditions inducing an elastic wave within the piezoelectric obstacle and a scattered acoustic wave traveling in a homogeneous unbounded domain. The system of PDE's used to model the problem combines the acoustic/elastic coupling conditions presented in \cite{HsSaSa:2016} for wave structure interaction, with the PDE's used in \cite{DeLaOh:2009} to describe the time evolution of the relevant variables. Aiming for a Finite Element discretization of the elastic and electric variables and a boundary element treatment of the acoustic wave, the system is translated into an integro-differential problem in the Laplace domain, the analysis is carried out following the techniques systematized in \cite{LaSa:2009a} and originated in the seminal work \cite{BaHa:1986a}. Galerkin discretization in space is used for all the variables, while Convolution Quadrature combined with time stepping are used for the time evolution.

We prove that the resulting fully discrete problem is well-posed and determine stability and error bounds with explicit time dependence for the time discretization based on second order backwards differentiation formulas (BDF2). A similar study with the backward Euler method is easy to obtain, while a Trapezoidal Rule CQ method is also available \cite{Banjai:2010}, but knowledge of the behavior of constants with respect to time is not known at current time. 

The paper is structured as follows. The general problem is presented in Section 2, where the time domain PDE model and the geometry are introduced along with the required notation and assumptions on physical parameters. Section 3 introduces the Laplace-transformed problem. Using standard BIE techniques we derive an equivalent integro-differential system and pose it variationally. The error equations satisfied by the resulting Galerkin space semi-discretization and the required elliptic projector are then presented. The core of the analysis is done in Section 4, where a slightly more general discrete system --encompassing both the discrete-in-space problem and the error equations-- is shown to be uniquely solvable by studying the variational formulation of an equivalent transmission problem; stability bounds are obtained in terms of the Laplace parameter $s$. The main results in the time domain are presented in Section 5, where the estimates obtained in the previous section are translated into the time domain and the system is fully discretized with BDF2-based Convolution Quadrature; for sufficiently smooth problem data second-order-in-time convergence is proven. Section 6 is dedicated to numerical experiments, some remarks pertaining the implementation of CQ and the coupling of boundary and finite elements are followed by experiments confirming convergence for the methods based in BDF2 and Trapezoidal Rule. Conclusions and possibilities for further studies are discussed in Section 7.
%
\paragraph{Background and notation.}
%
The following brackets
\[
(a,b)_\Omega:=\int_\Omega a\,b,
	\qquad
(\mathbf a,\mathbf b)_\Omega:=\int_\Omega\mathbf a\cdot\mathbf b,
	\qquad
(\mathbf A,\mathbf B)_\Omega:=\int_\Omega \mathbf A:\mathbf B,	
\]
will be used for scalar, vector and matrix-valued real $L^2$ inner products. (In the latter the colon denotes the Frobenius inner product of matrices.) When using complex-valued functions the brackets will remain bilinear and conjugation will be used as needed. If $\Omega$ is an open set with Lipschitz boundary or the complementary of such a set, then $H^1(\Omega)$ is the usual Sobolev space and its norm will be denoted by $\|\cdot\|_{1,\Omega}$. The space $H^{1/2}(\partial\Omega)$ is the trace space, while $H^{-1/2}(\partial\Omega)$ is the representation of its dual when $L^2(\partial\Omega)$ is identified with its dual space. The norms of  $H^{\pm1/2}(\partial\Omega)$, both for the real and complex valued cases, will be denoted $\|\cdot\|_{\pm 1/2,\partial\Omega}$. We will finally denote $\mathbf L^2(\Omega):=L^2(\Omega)^d$, $\mathbf H^1(\Omega):=H^1(\Omega)^d$, etc., and endow these spaces with the natural product norms. 
%
\section{Problem statement}
%
%
\paragraph{Geometric considerations.}
%
We will assume that the piezoelectric obstacle occupies a bounded, open and connected region $\Omega_-\subset\mathbb R^d$ with Lipschitz boundary $\Gamma$ which might not be connected. The boundary will be partitioned in two non-overlapping Dirichlet and Neumann parts, open relative to $\Gamma$ and such that $\Gamma = \overline\Gamma_D \cup \overline\Gamma_N$ and $\Gamma_D \neq\varnothing$. 

We will adopt the convention that the unit normal vector to the boundary $\boldsymbol\nu$ will be taken exterior to $\Omega_-$ and as such, it will always point towards the acoustic unbounded domain $\Omega_+ := \mathbb R^d \setminus \overline{\Omega}_-$.
%
\paragraph{Interior variables.}
%
In the interior domain the problem variables will be the elastic displacement field $\mathbf u$ and the electric potential $\psi$
\[
\mathbf u : \Omega_-\times[0,\infty) \longrightarrow \mathbb R^d,\qquad \psi : \Omega_-\times[0,\infty) \longrightarrow \mathbb R.
\]
Differential operators in the space variables will be unsubscripted meaning, for instance, that $\nabla\psi$ is the gradient in the space variables only. 
%
\paragraph{Physical parameters and tensors.}
%
We will use the following constitutive relations \cite{Tiersten:1969} to define the piezoelectric stress tensor $\boldsymbol \sigma$ and the electric displacement vector $\mathbf D$
\begin{equation}\label{eq:2.1}
\boldsymbol\sigma := \mathbf C \boldsymbol\varepsilon(\mathbf u) + \mathbf e \nabla\psi\,, \qquad \mathbf D := \mathbf e^\top \boldsymbol\varepsilon(\mathbf u) - \boldsymbol\epsilon\nabla\psi.
\end{equation}
In the above definition
\[
\boldsymbol\varepsilon(\mathbf u) := \tfrac{1}{2}\left(\nabla\mathbf u + \nabla\mathbf u ^{\top}\right)
\]
is the linear elastic strain tensor, $\mathbf C$, $\mathbf e$, and $\boldsymbol\epsilon$ are the elastic compliance, piezoelectric, and dielectric tensors respectively. They encode the electric and elastic properties of the material. For a real symmetric matrix $\mathbf M\in \mathbb R^{d\times d}_{sym}$ and for a vector, $\mathbf d \in \mathbb R^d$ we define
\begin{align*}
(\mathbf C(\mathbf x) \mathbf M)_{ij} := \,& \sum_{k,l}\mathbf C_{ijkl}(\mathbf x)\mathbf M_{kl}\,, \\
(\mathbf e(\mathbf x)\mathbf d)_{ij} := \,& \sum_{k}\mathbf e_{kij}\mathbf d_k\,, \\
(\mathbf e^\top(\mathbf x)\mathbf M)_{k} := \,& \sum_{ij}\mathbf e_{kij}\mathbf M_{ij}\,, \\
(\boldsymbol\epsilon(\mathbf x)\mathbf d)_i :=\,& \sum_{j}\boldsymbol\epsilon_{ij}(\mathbf x) \mathbf d_j.
\end{align*}
Due to physical considerations, these tensors exhibit the following symmetries \cite{GaKo:2001}:
\begin{equation}\label{eq:2.2}
\mathbf C_{ijkl} = \mathbf C_{jikl} = \mathbf C_{klij}\,,\quad \mathbf e_{lij} = \mathbf e_{lji}\,,\quad \boldsymbol\epsilon_{il}=\boldsymbol\epsilon_{li}.
\end{equation}
We will require that the components of the tensors are functions in $L^\infty(\Omega_-)$ and that for any symmetric matrix $\mathbf M\in \mathbb R^{d\times d}_{sym}$ and $\mathbf d \in \mathbb R^d$ there exist positive constants $c_0$ and $d_0$ such that for almost every $\mathbf x\in \Omega_-$
\[
\mathbf C(\mathbf x)\mathbf M:\mathbf M \geq c_0\mathbf M:\mathbf M \,, \qquad \boldsymbol\epsilon(\mathbf x) \mathbf d\cdot \mathbf d \geq d_0\mathbf d\cdot \mathbf d. 
\]
The density of the piezoelectric material will be denoted by $\rho_\Sigma$ and will be taken to be a strictly positive function in $L^\infty(\Omega_-)$.
%
\paragraph{Exterior domain.}
%
We will assume that the unbounded domain $\Omega_+$ surrounding the obstacle is filled with an homogeneous, isotropic and irrotational fluid with constant density $\rho_f$. Since the fluid is irrotational we can introduce a scalar velocity potential $v$ such that the actual fluid velocity can be expressed as $\mathbf v = \nabla v$. As is standard, the total acoustic wave field $v^{tot}$ will be the superposition of a given incident field $v^{inc}$ and an unknown scattered field $v$:
\[
v^{tot} = v^{inc}+v.
\]
Note that even if the domain $\Omega_-$ has a cavity the total wave there will contain a non-zero contribution of the incident wave, since $v^{inc}$ would be a solution of the problem in the absence of the scatterer. This decomposition is non-physical but mathematically meaningful. We will require that for $t\leq0$ the incident wave has not yet reached the obstacle. 
%
\paragraph{The PDE model.}
%
Under all the above considerations, and recalling the definitions \eqref{eq:2.1}, the interaction between the incident acoustic wave and the piezoelectric scatterer is governed by the following system of PDE's (see, for instance, \cite{Ihlenburg:1998,{DeLaOh:2009}}):

\begin{subequations}\label{eq:2.3}
\begin{alignat}{6}
\label{eq:2.3a}
\Delta v = &\, c^{-2}v_{tt} && \qquad \hbox{ in } \Omega_+\times[0,\infty), \\
\label{eq:2.3b}
\nabla\cdot\boldsymbol\sigma = &\,\rho_{\Sigma}\mathbf u_{tt}  &&\qquad \hbox{ in  }\Omega_-\times[0,\infty), \\
\label{eq:2.3c}
\nabla\cdot\mathbf D = &\, 0 &&\qquad \hbox{ in } \Omega_-\times[0,\infty),\\
\label{eq:2.3d}
\mathbf u_t\cdot\boldsymbol\nu + \partial_{\nu}v =&\, -\alpha_d && \qquad \hbox{ on } \Gamma\times[0,\infty),\\
\label{eq:2.3e}
\boldsymbol\sigma\,\boldsymbol\nu + \rho_f v_t\boldsymbol\nu =&\, -\rho_f (\beta_d)_t\boldsymbol\nu && \qquad \hbox{ on } \Gamma\times[0,\infty),\\
\label{eq:2.3f}
\mathbf D \cdot \boldsymbol\nu =\,& \eta_d &&\qquad \hbox{ on } \Gamma_N \times[0,\infty), \\
\label{eq:2.3g}
\psi = \,& \mu_d && \qquad \hbox{ on } \Gamma_D\times[0,\infty),
\end{alignat}
\end{subequations}
with homogeneous initial conditions for $v$, $v_t$, $\mathbf u$, and $\mathbf u_t$. 

The problem data is 
\begin{alignat*}{10}
\alpha_d :=\,&& \partial_\nu v^{inc}|_{\Gamma} :\,& \Gamma \times[0,\infty) \,&&\longrightarrow \mathbb R , && 
	\qquad \eta_d :\,&& \Gamma_N \times[0,\infty) \,&&\longrightarrow \mathbb R ,\\
\beta_d :=\,&& v^{inc}|_{\Gamma} :\,& \Gamma \times[0,\infty) \,&&\longrightarrow \mathbb R,
	&& \qquad \mu_d :\,&& \Gamma_D \times[0,\infty) \,&&\longrightarrow \mathbb R.
\end{alignat*}
This system can be given a rigorous form using causal distributions taking values in Sobolev spaces as in \cite{HsSaSa:2016,LaSa:2009a}, etc. For the time being will will deal only with the Laplace transform of this system and will come back to the time domain only at the time of giving stability and error estimates.

\begin{figure}\centering
\includegraphics[scale=.8]{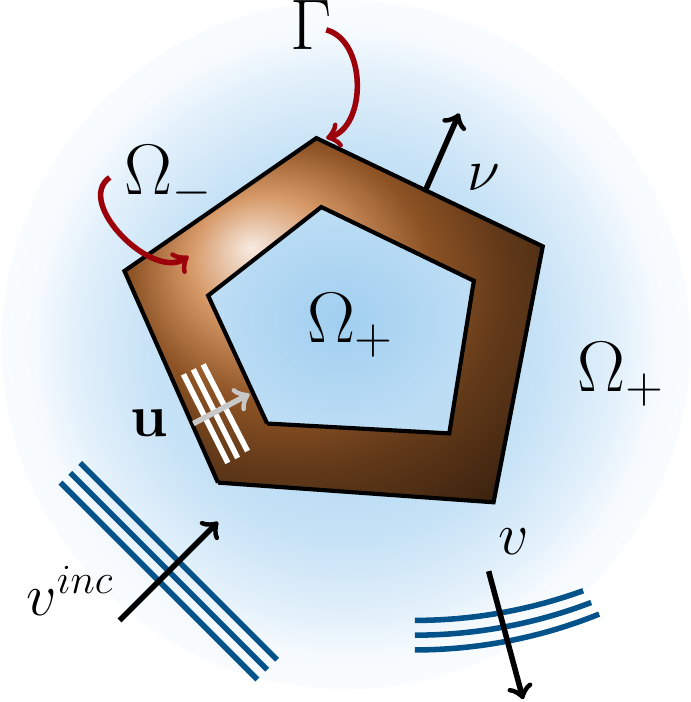}
\caption{{\footnotesize Scheme of the scattering geometry. The problem unknowns are the scattered acoustic wave $v$ (defined in the unbounded region $\Omega_+$), the induced elastic wave $\mathbf u$, and the electric potential $\psi$ (both defined inside the obstacle $\Omega_-$).The total elastic boundary $\Gamma$ is the disjoint union of the electric Dirichlet boundary $\Gamma_D$, and the electric Neumann boundary $\Gamma_N$ which are both open relative to $\Gamma$. The normal vectors $\boldsymbol \nu$ are exterior to the elastic domain and point towards the acoustic domain $\Omega_+$.}}\label{fig:1}
\end{figure}
%
\section{A Laplace domain semidiscrete problem}
%
%
In order to cast this problem rigorously we will need to define precisely the meaning of quantities on the boundary.  Given $v \in H^1(\mathbb R^d\setminus\Gamma)$, its interior, exterior, averaged, and difference traces will be denoted by:
\[
\gamma^-v, \quad \gamma^+v, \quad \ave{v}:=\smallfrac12(\gamma^-v+\gamma^+v),
\quad
\jump{\gamma v}:=\gamma^-v-\gamma^+v.
\]
For $(\mathbf u,\psi)\in \mathbf H^1(\Omega_-)\times H^1(\Omega_-)$ such that 
\[
\boldsymbol\sigma(\mathbf u,\psi)=\mathbf C\boldsymbol\varepsilon(\mathbf u) + \mathbf e\nabla\psi \in L^2(\Omega_-)^{d\times d}
\]
we define the weak interior traction field in analogy to the purely elastic normal stress with the formula
\[
\langle\boldsymbol\sigma\boldsymbol\nu,\gamma \mathbf w\rangle_\Gamma:=
	(\mathbf C \boldsymbol\varepsilon(\mathbf u),\boldsymbol\varepsilon(\mathbf w))_{\Omega_-} + (\mathbf e \nabla \psi,\boldsymbol\varepsilon(\mathbf w))_{\Omega_-} + (\nabla\cdot \boldsymbol\sigma,\mathbf w)_{\Omega_-}
\qquad\forall\mathbf w\in \mathbf H^1(\Omega_-).
\]
In a similar fashion, for $(\mathbf u, \psi) \in\mathbf H^1(\Omega_-)\times H^1(\Omega_-)$ such that 
\[
\mathbf D(\mathbf u,\psi)= \mathbf e^\top \boldsymbol\varepsilon(\mathbf u) - \epsilon\nabla\psi\in \mathbf L^2(\Omega_-),
\] 
the weak normal electric displacement will be defined by 
\[
\langle \mathbf D \cdot \boldsymbol\nu, \gamma w\rangle_\Gamma := (\boldsymbol\varepsilon(\mathbf u),\mathbf e \nabla w)_{\Omega_-}-(\boldsymbol\epsilon\nabla\psi,\nabla w)_{\Omega_-}+(\nabla\cdot\mathbf D,w)_{\Omega_-}\qquad\forall w\in H^1(\Omega_-).
\]
These are, respectively, elements of the dual spaces $\mathbf H^{-1/2}(\Gamma)$ and $H^{-1/2}(\Gamma)$. The $\Gamma$-subscripted angled bracket will be used to denote the duality product of either $H^{-1/2}(\Gamma)$ with $H^{1/2}(\Gamma)$ or $\mathbf H^{-1/2}(\Gamma)$ with $\mathbf H^{1/2}(\Gamma)$. 

To deal with the Dirichlet and Neumann conditions for the electric potential, we need additional notation. The restriction to the Dirichlet boundary of the trace of a function $u\in H^1(\Omega_-)$ will be denoted by $\gamma_Du:=\gamma u|_{\Gamma_D}$ and we define the function spaces
\begin{alignat*}{10}
H^{1/2}(\Gamma_D) :=\,& \left\{\gamma_D u : u \in H^1(\Omega_-) \right\}, &\qquad H^{1}_D(\Omega_-) :=\,& \left\{ u\in H^1(\Omega_-) : \gamma_D u=0 \right\}, \\
\widetilde H^{1/2}(\Gamma_N) :=\,&\left\{\gamma u|_{\Gamma_N}: u\in H^1_D(\Omega_-) \right\}, & \qquad H^{-1/2}(\Gamma_N) :=\,& \big(\widetilde H^{1/2}(\Gamma_N)\big)'.
\end{alignat*}
The angled bracket $\langle\cdot,\cdot\rangle_{\Gamma_N}$ should be understood as the duality pairing of $H^{-1/2}(\Gamma_N)$ with $\widetilde H^{1/2}(\Gamma_N)$. 
%
\paragraph{Laplace domain``dynamic" problem.}
%
From this point on, we will only consider the problem in the Laplace-domain and we will use the same symbol to represent a function and its Laplace transform without ambiguity. Let $s\in \mathbb C_+ :=\left\{s\in\mathbb C : \mathrm{Re} s >0\right\}$ and let
\[
(\alpha_d,\beta_d,\eta_d,\mu_d) \in H^{-1/2}(\Gamma)\times H^{1/2}(\Gamma)\times H^{-1/2}(\Gamma_N)\times H^{1/2}(\Gamma_D)
\]
be problem data. We look for $(v,\mathbf u,\psi)\in H^1(\Omega_+)\times\mathbf H^1(\Omega_-)\times H^1(\Omega_-)$ such that 
\begin{subequations}\label{eq:3.1}
\begin{alignat}{6}
\label{eq:3.1a}
\Delta v = &\, (s/c)^{2} v && \qquad \hbox{ in }\Omega_+, \\
\label{eq:3.1b}
\nabla\cdot\boldsymbol\sigma = &\,\rho_{\Sigma}s^2{\mathbf u}  &&\qquad \hbox{ in  }\Omega_-, \\
\label{eq:3.1c}
\nabla\cdot\mathbf D = &\, 0 &&\qquad \hbox{ in } \Omega_-,\\
\label{eq:3.1d}
s\gamma\mathbf u\cdot\boldsymbol\nu + \partial_{\nu}^+v =&\, -\alpha_d && \qquad \hbox{ on } \Gamma,\\
\label{eq:3.1e}
\boldsymbol\sigma\boldsymbol\nu + \rho_fs\gamma^+ v \boldsymbol\nu=&\, -\rho_fs\beta_d\boldsymbol\nu && \qquad \hbox{ on } \Gamma,\\
\label{eq:3.1f}
\mathbf D\cdot\boldsymbol\nu =\,& \eta_d &&\qquad \hbox{ on } \Gamma_N, \\
\label{eq:3.1g}
\gamma\psi = \,& \mu_d && \qquad \hbox{ on } \Gamma_D,
\end{alignat}
\end{subequations}
where $\boldsymbol\sigma$ and $\mathbf D$ are defined by \eqref{eq:2.1}.
%
\paragraph{Calder\'on calculus for the acoustic problem.}
%
The transmission problem
\begin{alignat*}{6}
\Delta v -s^2 v =0 &\qquad & \mbox{in $\mathbb R^d\setminus\Gamma$},\\
\jump{\gamma v}= \phi,\\
\jump{\partial_\nu v}= \lambda,
\end{alignat*}
with $\lambda\in H^{-1/2}(\Gamma)$ and $\phi\in H^{1/2}(\Gamma)$, is uniquely solvable, and its solution can be expressed using the integral representation formula
\begin{equation}\label{eq:3.2}
v=\mathrm S(s)\lambda -\mathrm D(s)\phi.
\end{equation}
The operators $\mathrm S$ and $\mathrm D$ are known as the single and double layer potentials respectively. Associated with the potentials, the following integral operators can be defined
\begin{alignat*}{6}
&\mathrm V(s):=\ave{\gamma \cdot}\mathrm S(s)=\gamma \mathrm S(s), & \qquad &
\mathrm K(s):=\ave{\gamma\cdot} \mathrm D(s),\\
&\mathrm K^\top(s):=\ave{\partial_\nu\cdot}\mathrm S(s),& \qquad &
\mathrm W(s):=-\ave{\partial_\nu\cdot} \mathrm D(s)=-\partial_\nu\mathrm D(s).
\end{alignat*}
We recall the following useful identities
\begin{equation}\label{eq:3.3}
\partial_{\nu}^{\pm} \mathrm S(s) = \mp \tfrac{1}{2}\mathrm I+ \mathrm K^\top(s) \,,\qquad \gamma^{\pm}\mathrm D(s) = \pm \tfrac{1}{2}\mathrm I + \mathrm K(s).
\end{equation}
%
\paragraph{A coupled integro-differential system.}
%
Having defined the notation and tools we will borrow from potential theory, we can now introduce the continuous integro-differential system that will be treated numerically later on. 

If $(v,\mathbf u, \psi)$ is a solution of \eqref{eq:3.1}, then $v$ can be represented as
\begin{equation}\label{eq:3.4}
v = \mathrm D(s/c)\phi - S(s/c)\lambda,
\end{equation}
where $\phi=\gamma^+v$ and $\lambda=\partial_{\nu}^+v$. Note that this representation can be extended to $\Omega_-$, yielding $v\equiv 0$ in $\Omega_-$. Therefore if we  use \eqref{eq:3.3} to  write $\gamma^-v$, we arrive at
\begin{equation}\label{eq:3.5}
\mathrm V(s/c)\lambda - \left(-\tfrac{1}{2}\mathrm I+\mathrm K(s/c)\right)\phi =  0 \qquad \hbox{ on }\;\Gamma.
\end{equation}
Additionally, \eqref{eq:3.1d} and the identities \eqref{eq:3.3} imply
\begin{equation}\label{eq:3.6}
\left(-\tfrac{1}{2}\mathrm I+\mathrm K^\top(s/c)\right)\lambda + \mathrm W(s/c)\phi - s\gamma\mathbf u\cdot\boldsymbol\nu  =\, \alpha_d \qquad \hbox{ on }\;\Gamma.
\end{equation}
Finally, \eqref{eq:3.1e} can be written in terms of $\phi$
\begin{equation}\label{eq:3.7}
\boldsymbol\sigma\boldsymbol\nu + \rho_fs\phi\boldsymbol\nu =\, -\rho_fs\beta_d\boldsymbol\nu \qquad \hbox{ on }\;\Gamma.
\end{equation}
Reciprocally, if the integro-differential equations \eqref{eq:3.5} through \eqref{eq:3.7} are satisfied, if $(\mathbf u, \psi,\lambda,\phi)$ satisfies \eqref{eq:3.1b}, \eqref{eq:3.1c}, \eqref{eq:3.1f}, and \eqref{eq:3.1g}, and we define $v$ with the representation formula \eqref{eq:3.4}, it follows that $\lambda = \partial^+_\nu v$, $\phi =\gamma^+v$, and we recover the system \eqref{eq:3.1}.
%
\paragraph{A variational formulation.}
%
We next define the elastodynamic bilinear form
\begin{subequations}\label{eq:3.99}
\begin{equation}\label{eq:3.99a}
b(\mathbf u,\mathbf w;s):=\, (\mathbf C \boldsymbol\varepsilon(\mathbf u),\boldsymbol\varepsilon(\mathbf w))_{\Omega_-}+s^2(\rho_\Sigma \mathbf u, \mathbf w)_{\Omega_-},
\end{equation}
and the coupled elastic-electric bilinear form
\begin{alignat}{6}
\mathcal B((\mathbf u,\psi),(\mathbf w,\varphi);s)
	&:=&& \, b(\mathbf u,\mathbf w;s) + (\mathbf e\nabla \psi,\boldsymbol\varepsilon(\mathbf w))_{\Omega_-} \nonumber  \\
	& && -(\boldsymbol\varepsilon(\mathbf u),\mathbf e\nabla \varphi)_{\Omega_-}+
(\boldsymbol\epsilon\nabla \psi,\nabla \varphi)_{\Omega_-}, \label{eq:3.99b} 
\end{alignat}
which is bounded in $\mathbf H^1(\Omega_-)\times H^1(\Omega_-)$. We also collect all the integral operators in \eqref{eq:3.5}-\eqref{eq:3.6} in a single matrix of operators
\begin{equation}\label{eq:3.99c}
\mathbb D(s):=
	\left[\begin{array}{cc}
		\mathrm V(s) & +\frac12\mathrm I-\mathrm K(s) \\
		-\frac12\mathrm I+\mathrm K^\top(s) & \mathrm W(s)
	\end{array}\right] : H^{-1/2}(\Gamma)\times H^{1/2}(\Gamma)\to
						H^{1/2}(\Gamma) \times H^{-1/2}(\Gamma).
\end{equation}
\end{subequations}
For the sake of notational simplicity, we will write $\mathbb D(s)(\lambda,\phi)$ for the action of $\mathbb D(s)$ on the column vector $(\lambda,\phi)^\top$.
We now present the continuous variational formulation of the problem which reads: 

Given data $(\alpha_d,\beta_d,\eta_d,\mu_d) \in H^{-1/2}(\Gamma)\times H^{1/2}(\Gamma)\times H^{-1/2}(\Gamma_N)\times H^{1/2}(\Gamma_D)$, find $(\mathbf u, \psi,\lambda,\phi)\in \mathbf H^1(\Omega_-)\times H^1(\Omega_-)\times H^{-1/2}(\Gamma)\times H^{1/2}(\Gamma)$ such that
\begin{subequations}\label{eq:3.8}
\begin{equation}
\gamma_D \psi =\mu_d,
\end{equation}
and for all $(\mathbf w, \varphi) \in \mathbf H^1(\Omega_-)\times H^1(\Omega_-)$ and $(\xi,\chi) \in  H^{-1/2}(\Gamma)\times H^{1/2}(\Gamma)$
\begin{alignat}{6}
\mathcal B((\mathbf u,\psi),(\mathbf w,\varphi);s)	
	& +\rho_fs\langle\phi,\gamma\mathbf w\cdot\boldsymbol\nu \rangle_{\Gamma} 
	&& = -\rho_fs\langle\beta_d,\gamma\mathbf w\cdot\boldsymbol\nu \rangle_{\Gamma}
		+   \langle\eta_d,\gamma \varphi\rangle_{\Gamma_N}, &  \\
-s\langle \gamma\mathbf u\cdot\boldsymbol\nu, \chi\rangle_{\Gamma}
	& +\langle\mathbb D(s/c)(\lambda,\phi),(\xi,\chi)\rangle_\Gamma
	&&=\langle\alpha_d, \chi\rangle_{\Gamma}.
\end{alignat}
\end{subequations}

%
\paragraph{Discrete formulation.}
%
In order to discretize the system \eqref{eq:3.8} a few definitions are in order. We consider finite dimensional subspaces
\[
\mathbf V_h \subseteq \mathbf H^1(\Omega_-),\quad V_h\subseteq H^1(\Omega_-), \quad V_{h,D}:=V_h\cap H^1_D(\Omega_-),\quad X_h\subseteq H^{-1/2}(\Gamma), \quad Y_h\subseteq H^{1/2}(\Gamma). 
\]
(Note that, following \cite{LaSa:2009a}, the theoretical treatment of the $s$-dependent discrete problem only uses that these spaces are closed. This has the advantage of simultaneously providing a well-posedness analysis of the continuous problem.) It will be assumed that the set 
\[
\mathbf M := \left\{ \mathbf m\in\mathbf H^1(\Omega_-): \boldsymbol\varepsilon(\mathbf m) = 0 \quad\forall \mathbf w\in\mathbf H^1(\Omega_-) \right\}
\]
of elastic rigid motions is always contained in $\mathbf V_h$. 
In the discrete case, the Dirichlet boundary condition will be approximated in the space $\gamma_D V_h := \left\{ \gamma_Dv^h: v^h\in V_h \right\}$. With all this in mind, we can now pose the discrete counterpart of \eqref{eq:3.8}. 

Given problem data $(\alpha_d,\beta_d,\eta_d,\mu_d^h) \in H^{-1/2}(\Gamma)\times H^{1/2}(\Gamma)\times H^{-1/2}(\Gamma_N)\times \gamma_DV_h$,  find $(\mathbf u^h, \psi^h,\lambda^h,\phi^h)\in \mathbf V_h\times V_h\times Y_h\times X_h$ such that
\begin{subequations}\label{eq:3.9}
\begin{equation}
\gamma_D \psi^h =\mu_d^h,
\end{equation}
and for all $(\mathbf w, \varphi) \in \mathbf V_h\times V_{h,D}$ and $(\xi,\chi)\in X_h\times Y_h$
\begin{alignat}{6}
\mathcal B((\mathbf u^h,\psi^h),(\mathbf w,\varphi);s)	
	& +\rho_fs\langle\phi^h,\gamma\mathbf w\cdot\boldsymbol\nu \rangle_{\Gamma} 
	&& = -\rho_fs\langle\beta_d,\gamma\mathbf w\cdot\boldsymbol\nu \rangle_{\Gamma}
		+   \langle\eta_d,\gamma \varphi\rangle_{\Gamma_N}, &  \\
\label{eq:3.9c}
-s\langle \gamma\mathbf u^h\cdot\boldsymbol\nu, \chi\rangle_{\Gamma}
	& +\langle\mathbb D(s/c)(\lambda^h,\phi^h),(\xi,\chi)\rangle_\Gamma
	&&=\langle\alpha_d, \chi\rangle_{\Gamma}.
\end{alignat}
\end{subequations}
A short-hand form of \eqref{eq:3.9c} can be given using polar sets. If $U_h$ is a subspace of the Hilbert space $U$, the expression $v\in U_h^\circ\subset U'$ ($v$ is in the polar set of $U_h$), will be shorthand for
\[
\langle v, u\rangle_{U'\times U} = 0 \qquad \forall u\in U_h.
\]
We can then shorten \eqref{eq:3.9c} as
\[
-(0,s\gamma\mathbf u^h\cdot\boldsymbol\nu)
	+\mathbb D(s/c)(\lambda^h,\phi^h)
	-(0,\alpha_d) \in X_h^\circ\times Y_h^\circ \equiv (X_h\times Y_h)^\circ.
\]
%
\paragraph{Trace liftings.}
%
By definition, the restriction of the trace to the Dirichlet boundary
\[
\gamma_D : H^1(\Omega_-) \longrightarrow H^{1/2}(\Gamma_D)
\]
is a surjective operator, and so is
\[
\gamma_{h,D}: = \gamma_D|_{V_h}:  V_h  \longrightarrow \gamma_D V_h.
\]
Note that there exists a bounded right-inverse of $\gamma_D$ (or lifting) which will be denoted by $\gamma_D^\dagger$. For the discrete counterpart, the existence of a right-inverse of $\gamma_{h,D}$ that is bounded uniformly in $h$ will be assumed (see \cite{DoSa:2003a}) and will be denoted $\gamma_{h,D}^\dagger$. 
%
\paragraph{An elliptic projector.}
%
A projection operator will be required in order to project the solid-electric component of the exact solution on the discrete space. Given $(\mathbf u,\psi,\mu_d^h) \in \mathbf H^1(\Omega_-)\times H^1_D(\Omega_-)\times\gamma_DV_h$ we will write $(\mathbf P_h \mathbf u, \mathrm P_h \psi) \in \mathbf V_h \times V_{h,D}$ to denote the pair satisfying
\begin{subequations}\label{eq:3.10}
\begin{equation}
\gamma_D\mathrm P_h \psi =\mu_d^h,
\end{equation}
the discrete variational equation
\begin{equation}\label{eq:3.10c}
\mathcal B((\mathbf P_h \mathbf u,\mathrm P_h\psi),(\mathbf w,\varphi);0)
	=\mathcal B(( \mathbf u,\psi),(\mathbf w,\varphi);0)
		\qquad \forall (\mathbf w,\varphi)\in \mathbf V_h\times V_{h,D},
\end{equation}
and the `grounding condition'
\begin{equation}
(\mathbf P_h\mathbf u,\mathbf m)_{\Omega_-} = (\mathbf u,\mathbf m)_{\Omega_-}
	\qquad \forall \mathbf m \in \mathbf M.
\end{equation}
\end{subequations}
Note that the bilinear form in \eqref{eq:3.10} does not contain the $s$-dependent term (we have set $s=0$), which is the kinetic part of the elastic-electric bilinear form $\mathcal B$. Note also that both $\mathbf P_h u$ and $\mathrm P_h\psi$ depend on $(\mathbf u,\psi)$ as well as on the discrete data $\mu_d^h$. We will keep the simplified (and somewhat misleading) notation for the sake of simplicity. The next lemma shows that this elliptic projection is quasioptimal.

\begin{lemma}\label{lem:3.1}
Problem \eqref{eq:3.10} is uniquely solvable and there exists $C>0$ such that
\begin{align*}
\|\mathbf u -\mathbf P_h\mathbf u\|_{1,\Omega_-}+\|\psi-\mathrm P_h\psi\|_{1,\Omega_-} \leq & C\Big(\inf_{\mathbf w \in \mathbf V_h}\|\mathbf u -\mathbf w\|_{1,\Omega_-}+\inf_{\varphi \in V_h}\|\psi -\varphi\|_{1,\Omega_-}\\
	& \qquad +\|\gamma_D\psi-\mu_d^h\|_{1/2,\Gamma}\Big).
\end{align*}
\end{lemma}
\begin{proof}
The bilinear form $\mathcal B(\cdot,\cdot;0)$ is coercive in the space \[
\{\mathbf u \in \mathbf H^1(\Omega_-) : (\mathbf u, \mathbf m)_{\Omega_-}=0\quad \forall \mathbf m \in \mathbf M\}\times H^1_D(\Omega_-),
\]
as follows from the second Korn and Poincar\'e-Friedrichs inequalities, since $\Gamma_D\neq\varnothing$. If $\{\mathbf m_i: i=1,\ldots,N_d\}$ is a basis for the space $\mathbf M$, then a simple compactness argument shows that the bilinear form
\[
\mathcal B((\mathbf u,\psi),
(\mathbf w,\varphi);0)+\sum_{i=1}^{N_d}(\mathbf u,\mathbf m_i)_{\Omega_-}(\mathbf w,\mathbf m_i)_{\Omega_-}
\]
is coercive in $H^1(\Omega_-)\times H^1_D(\Omega_-)$. Therefore, when $\mu^h_D=0$, it is simple to see that \eqref{eq:3.10} is just a Galerkin discretizations in $\mathbf V_h\times V_{h,D}$ of a coercive problem, and therefore a C\'ea estimate holds. The consideration of non-homogeneous boundary conditions, leading to an estimate like the one on the statement of the Lemma, can be approached using standard arguments based on the hypothesis of the existence of an $h$-uniform lifting of $\gamma_{h,D}$ \cite{DoSa:2003a}.
\end{proof}
%
\paragraph{Error equations.}
%
The error will be analyzed using the variables
\begin{alignat*}{6}
\mathbf e_{\mathbf u}^h :=\,&\mathbf P_h\mathbf u -\mathbf u^h, &\qquad
e_\psi^h :=\,&\mathrm P_h\psi - \psi^h, &\\
e_\lambda^h :=\,& \lambda -\lambda^h, &\qquad
e_\phi^h :=\,& \phi - \phi^h, &
\end{alignat*}
which satisfy
\begin{subequations}\label{eq:3.11}
\begin{equation}
\gamma_D e^\psi_h = 0,
\end{equation}
and for all $(\mathbf w, \varphi) \in \mathbf V_h\times V_{h,D}$
\begin{align}
\mathcal B((\mathbf e^h_{\mathbf u},e_\psi^h),(\mathbf w,\varphi);s)	
	 +\rho_fs\langle e_\phi^h,\gamma\mathbf w\cdot\boldsymbol\nu \rangle_{\Gamma} - s^2(\rho_\Sigma(\mathbf P_h\mathbf u-\mathbf u),\mathbf w)_{\Omega_-} 
	& = 0, &  \\
-(0,s\gamma\mathbf e_{\mathbf u}^h\cdot\boldsymbol\nu)
	 +\mathbb D(s/c)(e_\lambda^h,e_\phi^h) - (0,s\gamma(\mathbf P_h\mathbf u-\mathbf u)\cdot\boldsymbol\nu) 
	&\in X^\circ_h\times Y^\circ_h.
\end{align}
\end{subequations}
%
\section{Laplace domain analysis}
%
We consider a slightly more general problem from which both stability and error estimates for \eqref{eq:3.9} will be otained. Data are
\begin{subequations}
\begin{equation}\label{eq:4.1a}
(\alpha_d,\beta_d,\eta_d,\mu_d^h) \in H^{-1/2}(\Gamma)\times H^{1/2}(\Gamma)\times H^{-1/2}(\Gamma_N)\times \gamma_DV_h,
\end{equation}
and
\begin{equation}\label{eq:4.1b}
(\boldsymbol\theta_d,\theta_d,\lambda_d,\phi_d) \in \mathbf H^{1}(\Omega_-)\times L^2(\Gamma)\times H^{-1/2}(\Gamma)\times H^{1/2}(\Gamma),
\end{equation}
\end{subequations}
and we look for
\[
(\widehat{\mathbf u}^h,\widehat\psi^h)\in \mathbf V_h\times V_{h,D} \quad \text{and}\quad (\widehat\lambda^h,\widehat\phi^h)\in H^{-1/2}(\Gamma)\times H^{1/2}(\Gamma)
\]
such that for all $(\mathbf w,\varphi)\in\mathbf V_h\times V_{h,D}$ 
\begin{subequations}\label{eq:4.2}
\begin{align}
\label{eq:4.2a}
\gamma_D\widehat\psi^h = \,& \mu^h_d, \\
\nonumber
\mathcal B((\widehat{\mathbf u}^h,\widehat\psi^h),(\mathbf w, \varphi);s) + \rho_fs
\langle\widehat\phi^h,\gamma\mathbf w\cdot\boldsymbol\nu\rangle_\Gamma =\,& -\rho_fs\langle\beta_d,\gamma\mathbf w\cdot\boldsymbol\nu\rangle_\Gamma + \langle\eta_d,\gamma\varphi\rangle_{\Gamma_N}\\
\label{eq:4.2b}
	&  +s^2(\rho_\Sigma\boldsymbol\theta_d,\mathbf w)_{\Omega_-}, \\
\label{eq:4.2c}
-(0,s\gamma\widehat{\mathbf u}^h\cdot\boldsymbol\nu) + \mathbb D(s/c)(\widehat\lambda^h,\widehat\phi^h) - (0,\alpha_d+s\theta_d) \in\,& X^\circ_h\times Y^\circ_h, \\
\label{eq:4.2d}
(\widehat\lambda^h,\widehat\phi^h)-(\lambda_d,\phi_d) \in\,& X_h\times Y_h.
\end{align}
\end{subequations}
Note that if the first group of data \eqref{eq:4.1a} is set to be zero and we take
\[
(\boldsymbol\theta_d,\theta_d,\lambda_d,\phi_d) = (\mathbf P_h\mathbf u-\mathbf u, \gamma(\mathbf P_h\mathbf u-\mathbf u)\cdot\boldsymbol\nu,\lambda,\phi),
\] the system \eqref{eq:4.2} reduces to the error equations \eqref{eq:3.11}, while if the second group of data \eqref{eq:4.1b} is identically zero then the discrete equations \eqref{eq:3.9} are recovered.
%
\paragraph{Two equivalent problems.}
%
Using the Galerkin equations \eqref{eq:4.2} as the starting point, the analysis will proceed as in \cite{LaSa:2009a} by finding an equivalent transmission problem that can then be studied variationally. Solvability will then be established for the variational formulation and the stability constants will be obtained from the variational problem as well.

\begin{proposition}[Transmission problem]\label{prop:4.1}
If $(\widehat{\mathbf u}^h,\widehat\psi^h,\widehat\lambda,\widehat\phi)$ solves \eqref{eq:4.2}, and 
\begin{equation}\label{eq:4.99}
\widehat v^h := \mathrm D(s/c)\widehat\phi^h-\mathrm S(s/c)\widehat\lambda^h,
\end{equation}
then $(\widehat v^h,\widehat{\mathbf u}^h,\widehat \psi^h)\in H^1(\mathbb R^d\setminus\Gamma)\times\mathbf V_h\times V_{h,D}$ satisfies for all $(\mathbf w,\varphi)\in\mathbf V_h\times V_{h,D}$
\begin{subequations}\label{eq:4.3}
\begin{align}
\label{eq:4.3a}
-\Delta \widehat v^h +(s/c)^2\widehat v^h =\,& 0 \qquad \qquad \hbox{ in }\; \mathbb R^d\setminus\Gamma,\\
\label{eq:4.3b}
\gamma_D\widehat\psi^h =\,& \mu_d^h \qquad \qquad \hbox{ on }\; \Gamma, \\
\nonumber
\mathcal B((\widehat{\mathbf u}^h,\widehat\psi^h),(\mathbf w, \varphi);s) - \rho_fs
\langle\jump{\gamma\widehat v^h},\gamma\mathbf w\cdot\boldsymbol\nu\rangle_\Gamma =\,& -\rho_fs\langle\beta_d,\gamma\mathbf w\cdot\boldsymbol\nu\rangle_\Gamma + \langle\eta_d,\gamma\varphi\rangle_{\Gamma_N} \\
\label{eq:4.3c}
	&  +s^2(\rho_\Sigma\boldsymbol\theta_d,\mathbf w)_{\Omega_-},  \\
\label{eq:4.3d}
-(0,s\gamma\widehat{\mathbf u}^h\cdot\boldsymbol\nu) + (\gamma^-\widehat v^h,\partial^+_\nu \widehat v^h) - (0,\alpha_d+s\theta_d) \in\,& X^\circ_h\times Y^\circ_h, \\
\label{eq:4.3e}
(\jump{\partial_\nu\widehat v^h}+\lambda_d,\jump{\gamma \widehat v^h}+\phi_d)  \in\,& X_h\times Y_h.
\end{align}
\end{subequations}
Conversely, given a solution triplet $(\widehat v^h,\widehat{\mathbf u}^h, \widehat\psi^h)$ to \eqref{eq:4.3} and defining $\widehat\lambda^h:= -\jump{\partial_\nu \widehat v^h}$, and $\widehat\phi^h:=-\jump{\gamma \widehat v^h}$, the functions $(\widehat{\mathbf u}^h,\widehat\psi^h,\widehat\lambda^h,\widehat\phi^h)$ solve \eqref{eq:4.2}.
\end{proposition}
\begin{proof}
Given a solution of \eqref{eq:4.2}  and defining $\widehat v^h$ as in \eqref{eq:4.99}, it follows from the properties of the layer potentials that \eqref{eq:4.3a} is satisfied. Another consequence of this definition is that $\jump{\gamma\widehat v^h}=-\widehat\phi^h$ and $\jump{\partial_\nu\widehat v^h}=-\widehat\lambda^h$, which shows that equations \eqref{eq:4.3c} and \eqref{eq:4.3e} follow readily from \eqref{eq:4.2b}  and \eqref{eq:4.2d} by substitution of the above terms. Moreover, using the identities \eqref{eq:3.3} to compute $\gamma^-\widehat v^h$ and $\partial_\nu^+\widehat v^h$, it can be seen that \eqref{eq:4.2c} and \eqref{eq:4.3d} are equivalent.

The proof of the converse is very similar and requires only to observe that \eqref{eq:4.3a} allows for the layer potential representation of the acoustic field
\[
\widehat v^h = \mathrm D(s/c)\gamma^+\widehat v^h-\mathrm S(s/c)\partial_\nu^+ \widehat v^h.
\]
Thus, defining $\widehat\lambda^h$ and $\widehat\phi^h$ as in the statement all the above arguments can be repeated to show that equations \eqref{eq:4.2} hold.
\end{proof}

The system \eqref{eq:4.3} can now be treated variationally. To do this we introduce the space
\[
V_h^* := \left\{w\in H^1(\mathbb R^d\setminus\Gamma):\jump{\gamma w} \in Y_h,\;\hbox{ and }\; \gamma^-w\in X_h^\circ\right\}.
\]
The following proposition gives the equivalent variational formulation from which the solvability and stability bounds of the entire problem will be deduced.

\begin{proposition}[Variational formulation]\label{prop:4.2}
Consider the bilinear and linear forms
\begin{align*}
\mathcal A((v,\mathbf u,\psi),(w,\mathbf w,\varphi);s):=&\; (\nabla v,\nabla w)_{\mathbb R^d\setminus\Gamma} + (s/c)^2(v,w)_{\mathbb R^d\setminus\Gamma} & \\
& + \mathcal B((\mathbf u,\psi),(\mathbf w,\varphi);s) \\
		& + \rho_fs\langle\gamma\mathbf{u}\cdot\boldsymbol\nu,\jump{\gamma w} \rangle_{\Gamma} - \rho_fs\langle\jump{\gamma v},\gamma\mathbf{w}\cdot\boldsymbol\nu \rangle_{\Gamma},\\
	\\
\ell\left((w,\mathbf{w},\varphi);s\right):=\,& -\langle\lambda_d,\gamma^-w\rangle_{\Gamma} +\langle\alpha_d+s\theta_d,\jump{\gamma w}\rangle_{\Gamma} \\
	&  +s^2(\rho_\Sigma\boldsymbol\theta_d,\mathbf w)_{\Omega_-}-\rho_fs\langle\beta_d,\gamma\mathbf w\cdot\boldsymbol\nu\rangle_\Gamma \\
	& +\langle\eta_d,\gamma\varphi\rangle_{\Gamma_N}   .
\end{align*}
The system \eqref{eq:4.3} is equivalent to the problem of finding $(\widehat v^h,\widehat{\mathbf u}^h, \widehat\psi^h)\in H^1(\mathbb R^d\setminus\Gamma)\times \mathbf V_h \times V_{h,D}$ such that
\begin{subequations}\label{eq:4.4}
\begin{alignat}{6}
\label{eq:4.4a}
\gamma_D\widehat\psi^h =\;& \mu_d^h, &&\\
\label{eq:4.4b}
(\jump{\gamma \widehat v^h}+\phi_d,\gamma^-\widehat v^h) \in\;& Y_h\times X_h^\circ, &&\\
\label{eq:4.4c}
\mathcal A((\widehat v^h,\widehat{\mathbf u}^h,\widehat\psi^h),(w,\mathbf w,\varphi);s)=\;&\ell\left((w,\mathbf w,\varphi);s\right) \quad && \forall (w,\mathbf w,\varphi)\in  V_h^*\times\mathbf V_h \times V_{h,D} .
\end{alignat}
\end{subequations}
\end{proposition}
\begin{proof}
Given a solution $(\widehat v^h,\widehat{\mathbf u}^h, \widehat\psi^h)$ to \eqref{eq:4.3} we note that \eqref{eq:4.4b} is equivalent to the first component of \eqref{eq:4.3d} and the second component of \eqref{eq:4.3e}. Moreover, testing $\partial_\nu^+\widehat v$ with $\jump{\gamma w}$ for $w\in V_h^*$, we obtain
\begin{align}
\nonumber
\langle\partial_\nu^+\widehat v^h,\jump{\gamma w}\rangle_\Gamma =\,&
	 \langle \partial_\nu^-\widehat v^h,\gamma^-w\rangle_\Gamma - \langle\partial_\nu^+ \widehat v^h,\gamma^+w\rangle_\Gamma -\langle\jump{\partial_\nu\widehat v^h},\gamma^-w\rangle_\Gamma \\
\nonumber
=\,	& (\nabla\widehat v^h,\nabla w)_{\mathbb R^d\setminus\Gamma}+(\Delta \widehat v^h,w)_{\mathbb R^d\setminus\Gamma}-\langle\jump{\partial_\nu\widehat v^h},\gamma^-w\rangle_\Gamma\\
\nonumber
=\,	& (\nabla\widehat v^h,\nabla w)_{\mathbb R^d\setminus\Gamma}+(s/c)^2(\widehat v^h,w)_{\mathbb R^d\setminus\Gamma}-\langle\jump{\partial_\nu\widehat v^h}+\lambda_d,\gamma^-w\rangle_\Gamma+\langle\lambda_d,\gamma^-w\rangle_\Gamma \\
\label{eq:4.5}
=\,	& (\nabla\widehat v^h,\nabla w)_{\mathbb R^d\setminus\Gamma}+(s/c)^2(\widehat v^h,w)_{\mathbb R^d\setminus\Gamma}+\langle\lambda_d,\gamma^-w\rangle_\Gamma,
\end{align}
where we have used the definition of the weak normal derivative $\partial_\nu^\pm\widehat v^h$ in conjunction with equations \eqref{eq:4.3a}, \eqref{eq:4.3d} and the first component of \eqref{eq:4.3e}. Therefore, for $w\in V_h^*$ it follows from the second component of \eqref{eq:4.3e} and \eqref{eq:4.5} that 
\[
(\nabla\widehat v^h,\nabla w)_{\mathbb R^d\setminus\Gamma}+(s/c)^2(\widehat v^h,w)_{\mathbb R^d\setminus\Gamma}+\langle\lambda_d,\gamma^-w\rangle_\Gamma+\langle s\gamma\widehat{\mathbf u}^h\cdot\boldsymbol\nu + \alpha_d+s\theta_d,\jump{\gamma w}\rangle_\Gamma =0.
\]
This expression in combination with \eqref{eq:4.3c} are equivalent to \eqref{eq:4.4c}. 

To verify the converse statement, we expand the bilinear form in \eqref{eq:4.4c} and rewrite it in terms of the interior/exterior normal derivatives of $\widehat v^h$ and its Laplacian to show that
\begin{align*}
0 = \; &(\Delta \widehat v^h,w)_{\Omega_-} -(s/c)^2(\widehat v^h,w)_{\Omega_-} \\
& + \mathcal B((\widehat{\mathbf u}^h,\widehat\psi^h),(\mathbf w,\varphi);s) -s^2(\rho_\Sigma\boldsymbol\theta_d,\mathbf w)_{\Omega_-}- \rho_fs\langle\jump{\gamma \widehat v^h}+\beta_d,\gamma \mathbf{w}\cdot\boldsymbol\nu \rangle_{\Gamma} -\langle\eta_d,\gamma\varphi\rangle_{\Gamma_N}\\
&  + \langle\rho_fs\gamma\widehat{\mathbf u}^h\cdot\boldsymbol\nu+\alpha_d+s\theta_d,\jump{\gamma w} \rangle_{\Gamma} \\
& +\langle\jump{\partial_\nu\widehat v^h}+\lambda_d,\gamma^-w\rangle_{\Gamma}  .
\end{align*}
Once the equation is rewritten in this form, it is enough to notice that the mapping
\begin{align*}
V_h^*\times \mathbf V_h\times V_{h,D} \,& \longrightarrow V_h^*\times \mathbf V_h\times V_{h,D}\times X_h^\circ\times Y_h \\
(w,\mathbf w,\varphi) \; & \longmapsto (w,\mathbf w,\varphi,\gamma^-w,\jump{\gamma w})
\end{align*}
is surjective to conclude that every line of the above expression must vanish independently, which implies --line by line-- equations \eqref{eq:4.3a}, \eqref{eq:4.3c}, the second component of \eqref{eq:4.3d},  and the first component of \eqref{eq:4.3e}. The boundary condition \eqref{eq:4.3b} is given and the remaining two components of \eqref{eq:4.3d} and \eqref{eq:4.3e} are imposed strongly by the choice of function spaces.
\end{proof}
%
\paragraph{Well-posedness and stability.}
%
For $s\in \mathbb C_+$ we write
\[
\sigma:=\mathrm{Re}\,s>0\,, \qquad\underline\sigma:=\min\{\sigma,1\}.\]
To shorten some of the forthcoming expressions, we will denote:
\begin{alignat*}{6}
\| (v,\bff u,\psi)\|_1^2 
	:=\, &\rho_f\|\nabla v\|_{\mathbb R^d\setminus\Gamma}^2
		+\rho_fc^{-2}\| v\|_{\mathbb R^d}^2\\
	& +(\mathbf C\boldsymbol\varepsilon(\mathbf u),\boldsymbol\varepsilon(\overline{\mathbf u}))_{\Omega_-} +(\rho_\Sigma \bff u,\overline{\bff u})_{\Omega_-}+ (\boldsymbol\epsilon\nabla\psi,\nabla\overline{\psi})_{\Omega_-}.
\end{alignat*}
Following the program laid out in \cite{LaSa:2009a}, we define the energy norm
\begin{align*}
\triple{(v,\mathbf u,\psi)}_{|s|}^2:=\,& \rho_f\|\nabla v\|^2_{\mathbb R^d\setminus\Gamma} +\rho_fc^{-2}\|s v\|^2_{\mathbb R^d} \\
 	& +(\mathbf C \boldsymbol\varepsilon(\mathbf u),\boldsymbol\varepsilon(\overline{\mathbf u}))_{\Omega_-} +\|s\sqrt{\rho_\Sigma} \mathbf u\|^2_{\Omega_-}+ (\boldsymbol\epsilon\nabla\psi,\nabla\overline{\psi})_{\Omega_-},
\end{align*}
which includes kinetic and potential contributions from the acoustic and elastic fields, and the potential energy from the dielectric field. Notice that since $\Gamma_D\neq\varnothing$ this defines a norm in $V_h^*\times\mathbf V_h \times V_{h,D}$.
A simple computation shows that
\begin{equation}\label{eq:4.6}
\underline\sigma \| (v,\mathbf u,\psi)\|_1
	\le \triple{(v,\mathbf u,\psi)}_{|s|}
	\le \frac{|s|}{\underline\sigma}\| (v,\mathbf u,\psi)\|_1.
\end{equation}
\begin{proposition}[Well-posedness]\label{prop:4.3}
Problem \eqref{eq:4.4} is uniquely solvable for any
\begin{align*}
(\alpha_d,\beta_d,\eta_d,\mu_d^h) \in\;& H^{-1/2}(\Gamma)\times H^{1/2}(\Gamma)\times H^{-1/2}(\Gamma_N)\times \gamma_DV_h, \\
(\boldsymbol\theta_d,\theta_d,\lambda_d,\phi_d) \in\;& \mathbf H^{1}(\Omega_-)\times L^2(\Gamma)\times H^{-1/2}(\Gamma)\times H^{1/2}(\Gamma),
\end{align*}
and $s\in\mathbb C_+$. Moreover, there exists $C>0$ independent of $h$  and $s$ such that
\begin{subequations}\label{eq:4.7}
\begin{align}
\label{eq:4.7a}
\|(\widehat v^h,\widehat{\mathbf u}^h,\widehat\psi^h)\|_1 + \|\widehat\phi^h\|_{1/2,\Gamma}\leq\,&C\frac{|s|}{\sigma\underline{\sigma}^2} A(\mathrm{data},s),\\
\label{eq:4.7b}
\|\widehat\lambda^h\|_{-1/2,\Gamma} \leq\;& C\frac{|s|^{3/2}}{\sigma\underline{\sigma}^{3/2}} A(\mathrm{data},s),
\end{align}
\end{subequations}
where
\begin{align*}
A(\mathrm{data},s) := & \|\alpha_d\|_{-1/2,\Gamma} + \|s\beta_d\|_{1/2,\Gamma} + \|\eta_d\|_{-1/2,\Gamma} + \|s\mu_d^h\|_{1/2,\Gamma}\\
	& +\|s^2\boldsymbol\theta_d\|_{\Omega_-}+\|s\theta_d\|_{\Gamma}+\|\lambda_d\|_{-1/2,\Gamma}+\|s\phi_d\|_{1/2,\Gamma}.
\end{align*}
\end{proposition}
\begin{proof}
It is easy to check that
\[
\left|\mathrm{Re} \,\overline{s}\mathcal A\left((v,\mathbf u,\psi),\overline{(v,\mathbf u,\psi)};s\right)\right| = \sigma \triple{(v,\mathbf u,\psi)}_{|s|}^2. 
\]
This observation implies the existence and uniqueness of the solution by the Lax-Milgram lemma. In order to prove the stability bounds we first note that
\begin{align}
\label{eq:4.8}
\left|\mathcal A\left((v,\mathbf u,\psi),(w,\mathbf w,\varphi);s\right)\right| \leq\,& C_1\frac{|s|}{\underline\sigma}\triple{(w,\mathbf w,\varphi)}_{|s|}\|(v,\mathbf u,\psi)\|_{1}, \\
\nonumber
\left|\ell\left((w,\mathbf{w},\varphi);s\right)\right|\leq \,& \frac{C_2}{\underline\sigma}\triple{(w,\mathbf w,\varphi)}_{|s|}\Big(\|\alpha_d\|_{-1/2,\Gamma} + \|s\beta_d\|_{1/2,\Gamma} + \|\eta_d\|_{-1/2,\Gamma} \\
\label{eq:4.9}
	& \qquad\qquad\qquad +\|s^2\boldsymbol\theta_d\|_{\Omega_-}+\|s\theta_d\|_{\Gamma}+\|\lambda_d\|_{-1/2,\Gamma}\Big),
\end{align}
where $C_1$ and $C_2$ depend only on the geometry, and in the second inequality we have employed \eqref{eq:4.6} to bound the energy norm. Next, we pick liftings of the boundary data $\gamma^\dagger\phi_d\in H^1(\mathbb R^d\setminus\Gamma)$ and $\gamma_{h,D}^\dagger\mu_d^h \in V_h$ such that
\begin{subequations}\label{eq:4.10}
\begin{equation}
\label{eq:4.10a}
\gamma^-\gamma^\dagger \phi_d=0\,,\qquad -\gamma^+\gamma^\dagger\phi_d=\phi_d\,,\qquad \|\gamma^\dagger\phi_d\|_{1, \Omega_-}\leq C\|\phi_d\|_{1/2,\Gamma},
\end{equation}
\begin{equation}
\label{eq:4.10b}
\gamma\gamma^\dagger_{h,D} \mu_d^h=\mu_d^h\,,\qquad \|\gamma^\dagger_{h,D}\mu_d^h\|_{1, \Omega_-}\leq C\|\mu_d^h\|_{1/2,\Gamma}. \\ 
\end{equation}
\end{subequations}
Since $\widehat v^h+\gamma^\dagger\phi_d \in V_h^*$  and $\widehat\psi^h + \gamma^\dagger_{h,D}\mu_d^h \in V_{h,D}$, we can use \eqref{eq:4.4c} to show that
\[
\triple{(\widehat v^h+\gamma^\dagger\phi_d,\widehat{\mathbf u}^h,\widehat\psi^h + \gamma^\dagger_{h,D}\mu_d^h )}_{|s|}^2 
\phantom{\frac{|s|}{\sigma}\left|\mathcal A ((\widehat v^h+\gamma^\dagger\phi_d,\widehat{\mathbf u}^h,\widehat\psi^h + \gamma^\dagger_{h,D}\mu_d^h )\right|\phantom{\frac{|s|}{\sigma}\left|\mathcal A ((\widehat v^h+\gamma^\dagger\phi_d,\widehat{\mathbf u}^h,\widehat\psi^h + \gamma^\dagger_{h,D}\mu_d^h )\right|}}
\]
\vspace{-1cm}
\begin{align*}
\phantom{\left|\widehat{\mathbf u}^h,\widehat\psi^h + \gamma^\dagger_{h,D}\mu_d^h )\right|}  \leq & \frac{|s|}{\sigma}\left|\mathcal A ((\widehat v^h+\gamma^\dagger\phi_d,\widehat{\mathbf u}^h,\widehat\psi^h + \gamma^\dagger_{h,D}\mu_d^h ),(\widehat v^h+\gamma^\dagger\phi_d,\widehat{\mathbf u}^h,\widehat\psi^h + \gamma^\dagger_{h,D}\mu_d^h );s)\right| \\
 = & \frac{|s|}{\sigma}\left|\ell((\widehat v^h+\gamma^\dagger\phi_d,\widehat{\mathbf u}^h,\widehat\psi^h + \gamma^\dagger_{h,D}\mu_d^h );s)\right. \\
	& \;\quad\quad \left. + \mathcal A ((\gamma^\dagger\phi_d,\boldsymbol 0 , \gamma^\dagger_{h,D}\mu_d^h ),(\widehat v^h+\gamma^\dagger\phi_d,\widehat{\mathbf u}^h,\widehat\psi^h + \gamma^\dagger_{h,D}\mu_d^h );s)\right| \\
	\leq & \frac{|s|}{\sigma}  \triple{(\widehat v^h+\gamma^\dagger\phi_d,\widehat{\mathbf u}^h,\widehat\psi^h+\gamma^\dagger_{h,D}\mu_d^h)}_{|s|} \\
  & \times \bigg(\frac{C_2}{\underline{\sigma}}\Big( \|\alpha_d\|_{-1/2,\Gamma} + \|s\beta_d\|_{1/2,\Gamma} + \|\eta_d\|_{-1/2,\Gamma}
 +\|s^2\boldsymbol\theta_d\|_{\Omega_-} \\
 & \qquad+\|s\theta_d\|_{\Gamma}+\|\lambda_d\|_{-1/2,\Gamma}\Big)+ C_1\frac{|s|}{\underline\sigma}\Big(\|\gamma^\dagger\phi_d\|_{1, \Omega_-} + \|\gamma^\dagger_{h,D}\mu_d^h\|_{1,\Omega_-}\Big)\Bigg) \\
	\leq & C\frac{|s|}{\sigma\underline\sigma}  \triple{(\widehat v^h+\gamma^\dagger\phi_d,\widehat{\mathbf u}^h,\widehat\psi^h+\gamma^\dagger_{h,D}\mu_d^h)}_{|s|}A(\mathrm{data},s), \\
\end{align*}
where \eqref{eq:4.8}, \eqref{eq:4.9}, and \eqref{eq:4.10a} have been used. This implies
\begin{equation}
\triple{(\widehat v^h+\gamma^\dagger\phi_d,\widehat{\mathbf u}^h,\widehat\psi^h+\gamma^\dagger_{h,D}\mu_d^h)}_{|s|} \leq  C\frac{|s|}{\sigma\underline\sigma} A(\mathrm{data},s).
\label{eq:4.11}
\end{equation}
The inequality \eqref{eq:4.7a} follows from \eqref{eq:4.11} with an application of \eqref{eq:4.6}. The estimate \eqref{eq:4.7b} can be derived from \eqref{eq:4.11} by recalling that $\widehat\lambda^h=-\jump{\partial_\nu \widehat v^h}$ and applying \cite[Lemma 15]{LaSa:2009a} which states that, if $\Delta v-s^2 v=0$ in an open set $\mathcal O$ with Lipschitz boundary, then
\begin{equation}\label{eq:3.7}
\|\partial_\nu v\|_{-1/2,\partial\mathcal O}\leq C\left(\frac{|s|}{\underline\sigma}\right)^{1/2}
(\| s v\|_{\mathcal O}+\|\nabla v\|_{\mathcal O}).
\end{equation}
This finishes the proof.
\end{proof}

%
%
\section{Time domain estimates}
%
For the timed-domain estimates, data and solutions will be assumed to be in spaces of the form 
\[
W^k_+(\mathrm X):=\{ \xi\in \mathcal C^{k-1}(\mathbb R;\mathrm X)\,:\, \xi \equiv 0 \mbox{ in $(-\infty,0)$}, \xi^{(k)}\in L^1(\mathbb R;\mathrm X)\},
\]
for $k\ge 1$. We will also use the linear differential operator (cf. \cite{DoSa:2013})
\[
(\mathcal{P}_{k}f)(t) := \displaystyle\sum_{l=0}^{k} {k\choose l} f^{(l)}(t).
\]
The stability bounds and semi-discrete error estimates obtained in the previous section can be translated into the following time-domain results. Taking the second group of data \eqref{eq:4.1b} to be identically zero and setting
\[
(\boldsymbol\theta_d,\theta_d,\lambda_d,\phi_d) = (\mathbf P_h\mathbf u-\mathbf u, \gamma(\mathbf P_h\mathbf u-\mathbf u)\cdot\boldsymbol\nu,\lambda,\phi),
\] 
an application of \cite[Theorem 7.1]{DoSa:2013} combined with Proposition \ref{prop:4.3} yields the following result
\begin{corollary}[Stability in the time-domain]\label{cor:5.1}
Provided causal problem data
\[
(\alpha_d,\beta_d,\eta_d,\mu_d^h)\in W^3_+(H^{-1/2}(\Gamma))\times W^4_+(H^{1/2}(\Gamma))\times W^3_+(H^{-1/2}(\Gamma)) \times W^4_+(H^{1/2}(\Gamma))
\]
then  $\lambda^h, \phi^h, \mathbf u^h$ and $\psi^h$ are continuous causal functions of time and there exist $D_1, D_2>0$ such that, for $t\geq0$:
\begin{align*}
\|(v^h, \mathbf u^h,\psi^h)(t)\|_1 + \|\phi^h(t)\|_{1/2,\Gamma} \leq \,& \frac{D_1t^2}{t+1}\max\{1,t^2\}\int_0^t\|\mathcal{P}_{3}(\dot\alpha_d,\beta_d,\eta_d,\dot\mu_d^h)(\tau)\|_{\pm1/2,\Gamma}\;d\tau, \\
\|\lambda^h(t)\|_{-1/2,\Gamma} \leq \,& \frac{D_2t^{3/2}}{\sqrt{t+1}}\max\{1,t^{3/2}\}\int_0^t\|\mathcal{P}_{3}(\dot\alpha_d,\beta_d,\eta_d,\dot\mu_d^h)(\tau)\|_{\pm1/2,\Gamma}\;d\tau, 
\end{align*}
where $D_1$ and $D_2$ depend only on $\Gamma$.
\end{corollary}
We introduce the approximation error
\begin{align*}
a_h(t):=\;& \int_0^t\left( \|\mathcal P_3(\dot\phi-\Pi_h^Y\dot{\phi})(\tau)\|_{1/2,\Gamma} +\|\mathcal P_3(\lambda-\Pi_h^X\lambda)(\tau)\|_{-1/2,\Gamma} \right) \,d\tau\\
	& +\int_0^t\left(\|\mathcal P_3(\ddot{\mathbf u}-\mathbf P_h\ddot{\mathbf u})(\tau)\|_{\Omega_-} + \|\mathcal P_3(\dot{\mathbf u}-\mathbf P_h\dot{\mathbf u})(\tau)\|_{1,\Omega_-}\right)\,d\tau,
\end{align*}
where $\Pi_h^Y$ and $\Pi_h^X$ are the orthogonal projections onto $Y_h$ and $X_h$ respectively, and $\mathbf P_h$ is part of the elliptic projector defined in \eqref{eq:3.10}. Note that $\mathbf P_h\mathbf u$ depends on $\mathbf u$, $\psi$, and $\mu^h_d$ and that Lemma \ref{lem:3.1} states that 
\[
\|\mathbf u -\mathbf P_h\mathbf u\|_{1,\Omega_-}\leq C\big(\|\mathbf u -\Pi_{\mathbf V}^h\mathbf u\|_{1,\Omega_-}+\|\mathbf \psi -\Pi_{V}^h\psi\|_{1,\Omega_-}+\|\mu_d -\mu_d^h\|_{1/2,\Gamma}\big),
\]
where $\Pi_h^{\mathbf V}$ and $\Pi_h^V$ are the respective $H^1$ best approximation operators on $\mathbf V_h$ and $V_h$. Taking the data \eqref{eq:4.1a} as in Proposition \ref{prop:4.3} and applying \cite[Theorem 7.1]{DoSa:2013} we can prove the following result

\begin{corollary}[Semi-discrete error]\label{cor:5.2}
If  the solution $(\lambda,\phi,\mathbf u,\psi)$ is causal and belongs to
\[
W^3_+(H^{-1/2}(\Gamma))\times W^4_+(H^{1/2}(\Gamma))\times \big[W^5_+(\mathbf H^1(\Omega_-))\cap W^4_+(\mathbf L^2(\Omega_-)) \big]\times \big[W^5_+(H^1(\Omega_-))\cap W^4_+( L^2(\Omega_-))\big]
\]
then, for every $t\geq0$
\begin{align*}
\|(\mathbf e^h_v,\mathbf e^h_{\mathbf u},e_{\psi}^h)(t)\|_1 + \|e_\phi^h(t)\|_{1/2,\Gamma} \leq \frac{D_1t^2}{t+1}\max\{1,t^2\}a_h(t), \\
\|(e_\lambda^h)(t)\|_{1/2,\Gamma} \leq \frac{D_2t^{3/2}}{\sqrt{t+1}}\max\{1,t^{3/2}\}a_h(t),
\end{align*}
where $D_1, D_2>0$ depend only on $\Gamma$,
\[
e_v^h:= \mathcal D*(\phi-\phi^h)-\mathcal S*(\lambda-\lambda^h)=v- \mathrm D*\phi^h+\mathrm S*\lambda^h.
\]
Here $\mathcal D$ and $\mathcal S$ are the operator-valued causal distributions whose Laplace transforms are $\mathrm D(s/c)$ and $\mathrm S(s/c)$, and $*$ is the symbol for distributional convolution in the time variable.
\end{corollary}
%
\paragraph{Convolution Quadrature and time-stepping.}
%
Convolution Quadrature (CQ) was developed by Christian Lubich in the late 80's and early 90's \cite{HaLuSc:1985,Lubich:1988a,Lubich:1988b,Lubich:1994} as a way to approximate causal convolutions and convolution equations based on the knowledge of the Laplace transform of the convolution kernel and time domain data.

Since then it has been enriched greatly by works like \cite{ScLoLu:2006,CaCuPa:2007,HaKrSa:2009, Banjai:2010,BaLu:2011,BaLuMe:2011,BaLuSa:2015} and --due to its stability properties, the advantage of requiring only Laplace-domain fundamental solutions and the possibility to take damping effects into account with relative ease -- has become one of the preferred tools for the numerical analysis and simulation of evolutionary problems arising from wave propagation and diffusion. A thorough review of results and properties of CQ applied to boundary integral equations can be found in \cite{BaSc:2012}, while \cite{HaSa:2014} gives a detailed explanation of the computational and algorithmic subtleties involved in its implementation.

In the heart of every CQ implementation lies an ODE solver which determines its analytic and convergence properties and gives rise to different families of CQ algorithms. We will focus on BDF2 time-stepping for the analysis but will give numerical experiments for both BDF2 and Trapezoidal Rule-based methods. The integral equations will be treated numerically with CQ, while the Finite Element discretization of the elastic displacement and electric potential will be discretized in time using BDF2 and TR time-steppers.

The key resut that will allow us to carry out all the time domain analysis using CQ based tools --even if our computational implementation involves traditional time stepping for the finite element discretization-- is that this split treatment of different parts of a system is equivalent to the application of CQ globally, as long as the time stepping method used for the FEM part coincides with the one giving rise to the CQ algorithm in use (see \cite[Proposition 12]{LaSa:2009a}, \cite{HaSa:2016}). 

The approximation error between the fully discrete solution $(v^h_\kappa,\mathbf u^h_\kappa,\psi^h_\kappa)$ obtained using BDF2-CQ with a time step size $\kappa$ and the semi-discrete approximation $(v^h,\mathbf u^h,\psi^h)$ can be estimated from Propostion \ref{prop:4.3} using \cite[Proposition 4.6.1]{Sayas:2014} (a slight variant of a result in \cite{Lubich:1994}). 
\begin{corollary}\label{cor:5.3}
Let $\ell=6$ and $(\alpha_d,\beta_d,\eta_d,\mu_d^h)$ be causal problem data satisfying
\[
(\alpha_d,\beta_d,\eta_d,\mu_d^h)\in W^{\ell +1}_+(H^{1/2}(\Gamma))\times W^{\ell}_+(H^{-1/2}(\Gamma))\times W^{\ell}_+(H^{-1/2}(\Gamma)) \times W^{\ell}_+(H^{1/2}(\Gamma)).
\] 
For $t\geq0$, the difference between the semi-discrete solution and fully discrete solution obtained using BDF2-based Convolution Quadrature is bounded like
\[
\|(v^h,\mathbf u^h,\psi^h)(t)-(v^h_\kappa,\mathbf u^h_\kappa,\psi^h_\kappa)(t)\|_1 \leq D(1+t^2)\kappa^2\int_0^t\|(\dot\alpha_d,\beta_d,\eta_0,\mu_0^h)^{(\ell)}(\tau)\|_\Gamma\;d\tau,
\]
where $D$ depends only on $\Gamma$. Reduced convergence of order 2/3 is achieved for $\ell = 3$. 
\end{corollary}
%
\section{Numerical experiments}
%
In order to test the convergence results proven in the previous section, the formulation was implemented using standard Lagrangian finite elements for the elastic and electric fields and Galerkin boundary elements for the acoustic field. We take $\mathbf V_h$ and $V_h$ to be continuous $\mathcal P_k$ finite elements on a triangular mesh of $\Omega_-$. On the inherited mesh on $\Gamma$, we consider the space $X_h$ of discontinuous piecewise $\mathcal P_{k-1}$ and the space $Y_h$ of continuous piecewise $\mathcal P_k$ functions.
%
\paragraph{About the implementation.}
%
One of the advantages of the formulation we propose is that it lends itself to a highly modular implementation, in the sense that pre-existing FEM code for piezoelectricity and BEM code for acoustics can be used to solve the coupled problem in the frequency domain without any modification. The only requirement is the addition of a ``discrete trace" which translates boundary FEM degrees of freedom into BEM degrees of freedom. Formally, the structure of the discrete system \eqref{eq:3.8} can be represented as
\[
\left(\begin{array}{cc} \mathrm{\mathbf{FEM}}(s) & s\rho_f(\mathrm N\Gamma)_h^\top \\ -s\rho_f(\mathrm N\Gamma)_h &\rho_f\mathrm{\mathbf{BEM}}(s) \end{array}\right)
\left(\begin{array}{c} \left(\begin{array}{c}\mathbf u^h \\ \psi^h \end{array}\right) \\[2ex] \left(\begin{array}{c} \lambda^h \\ \phi^h 
\end{array}\right) \end{array}\right) = 
\left(\begin{array}{c} 
\left(\begin{array}{c}
	-s\rho_f\Gamma_h^\top\beta^h \\ \eta^h 
\end{array}\right) \\[2ex]
\left(\begin{array}{c}
	 0 \\ \rho_f\alpha^h
\end{array}\right)
 \end{array}\right),
\]  
where: (a) the finite element block $\mathbf{FEM}(s)$ contains sparse $s$-independent elastic stiffness, material mass, piezoelectric, and electric stiffness-like matrices, the material mass matrix being multiplied by $s^2$ (see \eqref{eq:3.99a} and \eqref{eq:3.99b});   (b) 
the boundary element block $\mathbf{BEM}(s)$ contains Galerkin discretizations of the operators of the Calder\'on projector (see \eqref{eq:3.99c}); (c)  the sparse matrix $(\mathrm N\Gamma)_h$ corresponds to the discretization of the bilinear form $\mathbf V_h\times Y_h \ni (\mathbf u^h,\chi^h)\mapsto \langle\mathbf u^h\cdot\boldsymbol\nu,\chi^h\rangle_\Gamma$ with added zero blocks for the interactions of all other spaces. We note that the trace space for $\mathbf V_h$ is a vector-valued version of $Y_h$, which means that, apart from rearrangements of degrees of freedom (and possible changes of local polynomial bases), the only matrix connecting the BEM and FEM codes is simple to implement.

In a similar way, the transition from Laplace domain to time domain can  be done in a modularly, either by implementing a CQ routine that inverts the full operator matrix, or a time stepping routine where $s$ is replaced with a discrete approximation of the differentiation operator, or using a Schur complement strategy as was first suggested in \cite{BaSa:2009} and outlined in \cite{HaSa:2016} for a purely acoustic system or as in \cite{HsSaSa:2016} for a coupled acoustic/elastic problem. The latter approach, which results in a decoupling of the boundary integral part of the system, is well suited for parallelization and was the chosen strategy for the following numerical experiments. 
%
\paragraph{Geometric setup and physical parameters.}
%
In all the convergence studies (frequency and time domain), the rectangle $\Omega_-:=(1,\,3)\times (1,\, 2) \subset \mathbb R^2$ was used as the piezoelectric domain. The double-indices used in our general presentation of tensor in the piezoelectric domains will be reduced to a single index using the simple convention:
\[
(1,1) \leftrightarrow 1 \qquad (2,2) \leftrightarrow 2
\qquad (1,2) \leftrightarrow 3.
\]
(By symmetry, the pair (2,1) can be avoided in the tensor representations.) We choose the following constant Lam\'e parameters, mass density, and acoustic speed of sound:
\begin{subequations}\label{eq:00}
\begin{equation}\label{eq:00a}
\lambda = 2, \quad \mu= 3, \quad \rho=5, \quad c=1.
\end{equation}
We will use Young's modulus and Poisson's ratio
\begin{equation}\label{eq:00b}
E := \frac{2\mu(1+\lambda)}{2\mu+\lambda}, \qquad \nu := \frac{\lambda}{2\mu+\lambda}
\end{equation}
to express the entries of the elastic compliance tensor $\mathbf C$
\begin{equation}\label{eq:00c}
\mathbf C_{11} = \mathbf C_{22}= \frac{E}{1-\nu^2},\quad \mathbf C_{33} = \frac{E}{2(1+\nu)}, \quad \mathbf C_{12}= \frac{E\nu}{1-\nu^2},\quad \mathbf C_{13}=\mathbf C_{23} = 0.
\end{equation}
For the piezoelectric tensor $\mathbf e$ the values used were
\begin{equation}\label{eq:00d}
\mathbf e_{11}=\mathbf e_{22}=\mathbf e_{33}= 1,\quad \mathbf e_{12} = \mathbf e_{13}=\mathbf e_{23} = 5,
\end{equation}
while for the dielectric tensor $\boldsymbol \epsilon$ the entries were
\begin{equation}\label{eq:00e}
\boldsymbol \epsilon_{11}=\boldsymbol \epsilon_{22}=4 ,\quad \boldsymbol \epsilon_{12}=1.
\end{equation}
\end{subequations}
We take $\Gamma_D=\Gamma$ and $\Gamma_N=\varnothing$. 
%
\paragraph{Convergence studies in the frequency domain.}
%
The elastic plane pressure wave
\[
\mathbf{u}(\mathbf{x}) = e^{-sc_L\mathbf{x}\cdot\mathbf{d}}\mathbf{d}, \quad \mathbf{d}=\left(\tfrac{1}{\sqrt{2}},\tfrac{1}{\sqrt{2}}\right), \quad c_L = \sqrt{\tfrac{2\mu+\lambda}{\rho}}
\]
 with $s=-2.5\imath$ was imposed as a solution alongside the electric field
\[
\psi(\mathbf x) = x_1^3+x_1^3x_2-3x_1x_2^2-\tfrac{1}{3}x_2^3.
\]
In the acoustic domain, the cylindrical acoustic wave 
\[
v(\mathbf x)=\tfrac{\imath}{4} H^{(1)}_0(\imath s |\mathbf x-\mathbf x_0|), \qquad \mathbf x_0 = (2,1.5)\in \Omega_-
\]
was used. Right-hand sides are added to equations \eqref{eq:3.1b} and \eqref{eq:3.1c}
\[
\nabla\cdot\boldsymbol\sigma=\rho_\Sigma s^2 \mathbf u+f_1
\qquad
\nabla\cdot\mathbf D=f_2
\]
so that $(\mathbf u,\psi)$ is a solution. Note that both $f_1$ and $f_2$ are independent of the frequency $s$, due to the fact that we have chosen $\mathbf u$ to be an elastic plane wave. Boundary data for $\psi$ and transmission data in 
\eqref{eq:3.1d}-\eqref{eq:3.1e} are built so that the equations are satisfied.

The experiment was ran using $k=1,2$ for the $\mathcal P_k$ finite elements and $\mathcal P_k/\mathcal P_{k-1}$ boundary elements. The acoustic wave was sampled in twenty random points in the exterior of the piezoelectric domain and compared to the exact solution, using the maximum discrepancy as the measure of the acoustic error $E^{v}_{h}$. For the elastic and electric unknowns both $L^2(\Omega_-)$ and $H^1(\Omega_-)$ errors were computed. Tables \ref{tab:1} and \ref{tab:2} as well as Figure \ref{fig:2} show the outcome of the convergence tests.
\begin{table}[h]\centering
\scalebox{0.8}{
\begin{tabular}{ccccccccccccc}
\hline
\multicolumn{3}{|c|}{$k=1$} & \multicolumn{4}{|c|}{$L^2(\Omega_-)$} & \multicolumn{4}{|c|}{$H^1(\Omega_-)$} \\
\hline
\multicolumn{1}{|c|}{$h$} & \multicolumn{1}{c|}{$E^{v}_{h}$} & \multicolumn{1}{c|}{e.c.r.} & \multicolumn{1}{c|}{$E^{\mathbf u}_{h}$} & \multicolumn{1}{c|}{e.c.r.} & \multicolumn{1}{c|}{$E^{\psi}_{h}$} & \multicolumn{1}{c|}{e.c.r.} & \multicolumn{1}{c|}{$E^{\mathbf u}_{h}$} & \multicolumn{1}{c|}{e.c.r.} & \multicolumn{1}{c|}{$E^{\psi}_{h}$} & \multicolumn{1}{c|}{e.c.r.} \\ \hline
 0.2  	&   7.110 E-2  & ---    & 1.167 E-1  & ---   &  4.140 E-2  & ---    & 7.835 E-1  & ---    & 1.718     & ---    \\ \hline
 0.1 	&  1.760 E-2  & 2.014  & 3.146 E-2  & 1.891 &  1.047 E-2  & 1.984  & 2.646 E-1  & 1.566  & 8.544 E-1    & 1.007  \\ \hline
 0.05 &   4.615 E-3  & 1.931  & 8.138 E-3  & 1.951 &  2.632 E-3  & 1.991  & 9.372 E-2  & 1.497  & 4.263 E-1 & 1.003  \\ \hline
 0.025 &   1.171 E-3  & 1.978  & 2.059 E-3  & 1.983 &  6.599 E-4  & 1.996  & 3.734 E-2  & 1.327  & 2.130 E-1 & 1.001  \\ \hline     
\end{tabular}
}
\caption{{\footnotesize Relative errors and estimated convergence rates in the time frequency domain with $\mathcal P_1$ finite elements and $\mathcal P_1/\mathcal P_0$ boundary elements. $h$ represents the maximum length of the panels used to discretize the boundary.}}\label{tab:1}
\end{table}
\begin{table}[h]\centering
\scalebox{0.8}{
\begin{tabular}{ccccccccccccc}
\hline
\multicolumn{3}{|c|}{$k=2$} & \multicolumn{4}{|c|}{$L^2(\Omega_-)$} & \multicolumn{4}{|c|}{$H^1(\Omega_-)$} \\
\hline
\multicolumn{1}{|c|}{$h$} & \multicolumn{1}{c|}{$E^{v}_{h}$} & \multicolumn{1}{c|}{e.c.r.} & \multicolumn{1}{c|}{$E^{\mathbf u}_{h}$} & \multicolumn{1}{c|}{e.c.r.} & \multicolumn{1}{c|}{$E^{\psi}_{h}$} & \multicolumn{1}{c|}{e.c.r.} & \multicolumn{1}{c|}{$E^{\mathbf u}_{h}$} & \multicolumn{1}{c|}{e.c.r.} & \multicolumn{1}{c|}{$E^{\psi}_{h}$} & \multicolumn{1}{c|}{e.c.r.} \\ \hline
 0.2  	&   5.545 E-5  & ---    & 3.542 E-4  & ---   &  3.927 E-4  & ---    & 1.350 E-2  & ---    & 1.805 E-2     & ---    \\ \hline
 0.1 	&   4.161 E-6  & 3.736  & 3.949 E-5  & 3.024 &  4.872 E-5  & 3.024  & 3.083 E-3  & 2.130  & 4.450 E-3     & 2.020  \\ \hline
 0.05 &   3.146 E-7  & 3.725  & 4.555 E-6  & 3.116 &  5.991 E-6  & 3.010  & 7.153 E-4  & 2.108  & 1.105 E-3 & 2.009  \\ \hline
 0.025 &   2.379 E-8  & 3.725  & 5.455 E-7  & 3.062 &  7.463 E-7  & 3.005  & 1.710 E-4  & 2.064  & 2.753 E-4 & 2.005  \\ \hline      
\end{tabular}
}
\caption{{\footnotesize Relative errors and estimated convergence rates in the time frequency domain with $\mathcal P_2$ finite elements and $\mathcal P_2/\mathcal P_1$ boundary elements. $h$ represents the maximum length of the panels used to discretize the boundary.}}\label{tab:2}
\end{table}
\begin{figure}[h]\centering
\includegraphics[height =5cm]{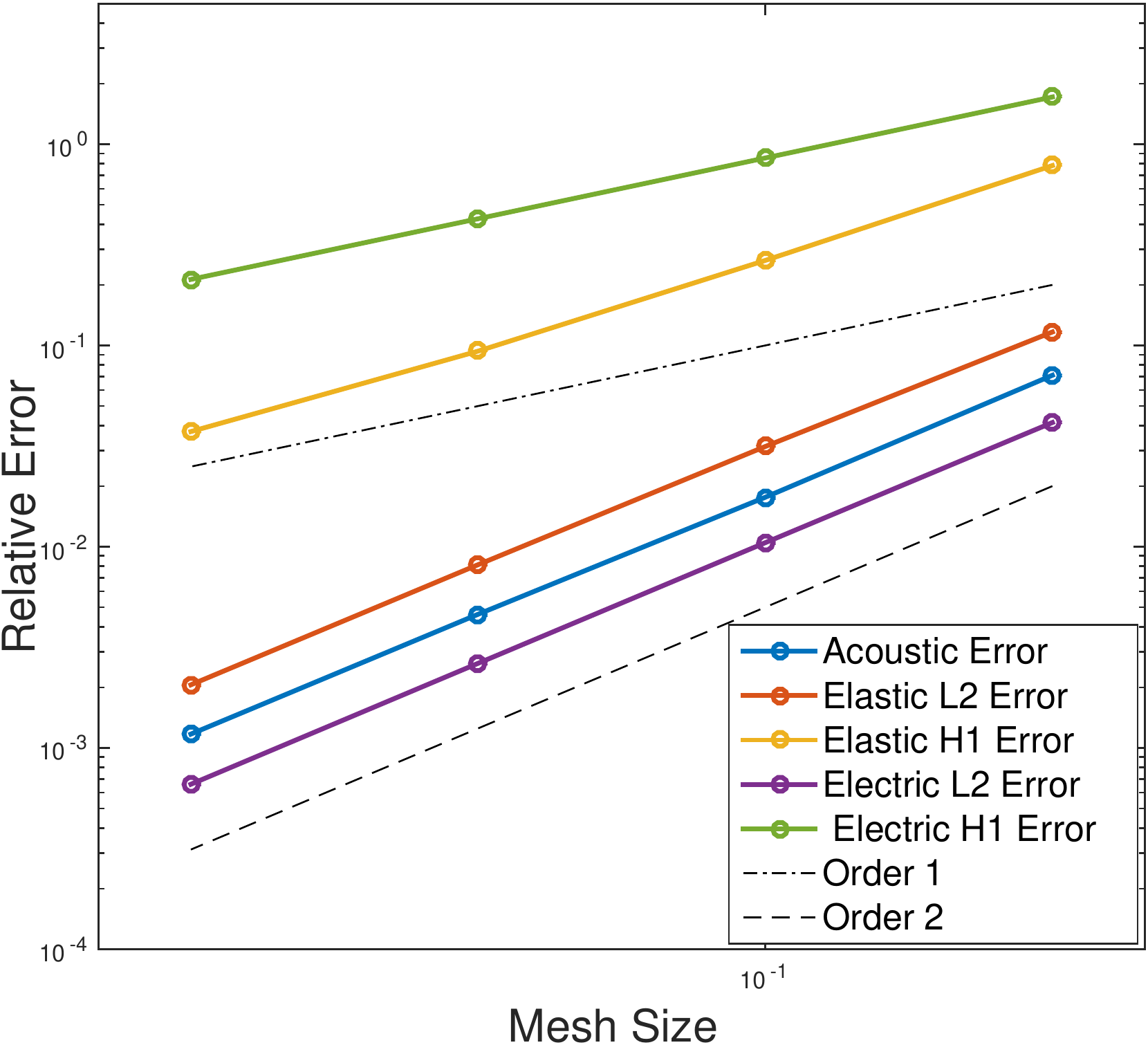}\qquad
\includegraphics[height =5cm]{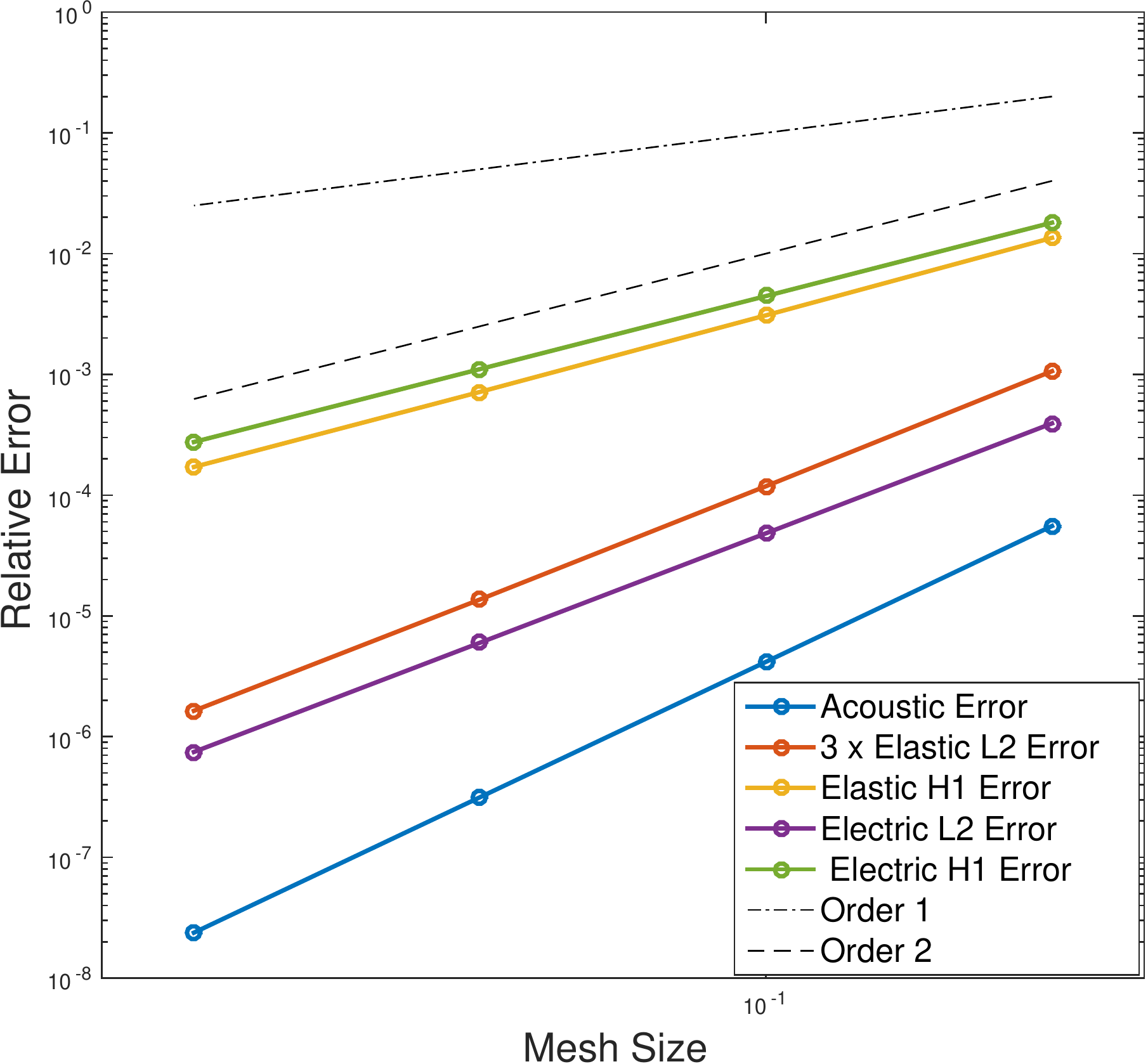}
\caption{{\footnotesize Convergence studies for the frequency domain problem are shown for $\mathcal P_1/\mathcal P_0$ boundary elements and $\mathcal P_1$ finite elements (left) and $\mathcal P_2/\mathcal P_1$ boundary elements and $\mathcal P_2$ finite elements (right).}}\label{fig:2}
\end{figure} 
%
\paragraph{Convergence studies in the time domain.}
%
Experiments were carried out using matching time stepping (for the FEM part) and CQ (for the BEM part) based on both Trapezoidal Rule and BDF2 for time evolution. The fully discrete method based on the trapezoidal rule can analyzed in the same way as BDF2, using results from \cite{Banjai:2010}. Note that the only difference is the lack of knowledge of how the error constants depend on the time variable. 

Just as in the frequency domain case, the rectangle $\Omega_-:=(1,\,3)\times(1,\, 2)\subset \mathbb R^2$ was used as the piezoelectric domain where the elastic plane pressure wave
\[
\mathbf{u}(\mathbf{x},t) = \mathcal{H}(c_Lt-\mathbf{x}\cdot\mathbf{d})\sin\left(3(c_Lt-\mathbf{x}\cdot\mathbf{d})\right)\mathbf{d}, \quad \mathbf{d}=\left(\tfrac{1}{\sqrt{2}},\tfrac{1}{\sqrt{2}}\right), \quad c_L = \sqrt{\tfrac{2\mu+\lambda}{\rho}},
\]
and the causal electric field
\[
\psi(\mathbf x,t) = \mathcal{H}(t)(x_1^3+x_1^3x_2-3x_1x_2^2-\tfrac{1}{3}x_2^3),
\] 
were imposed. In the acoustic domain, the cylindrical acoustic wave
\[
v(\mathbf x,t)=\mathcal{L}^{-1}\left\{\imath H^{(1)}_0(\imath s |\mathbf x-\mathbf x_0|)\,\mathcal{L}\{\mathcal{H}(t)\sin(2t)\} \right\},
\]
centered at $\mathbf x_0 = (2,1.5)$, was imposed. In all cases $\mathcal{H}$ is the piecewise polynomial approximation to Heaviside's step function
\[
\mathcal H(t):= t^5(1-5(t-1)+15(t-1)^2-35(t-1)^3+70(t-1)^4-126(t-1)^5)\chi_{[0,1]}+\chi_{[1,\infty)}.
\]
Analogously to the frequency domain experiments, right hand sides were added so that $(\mathbf u,\psi)$ are solutions to the system
and the appropriate Dirichlet data was sampled at the boundary using Equations \eqref{eq:2.3d}. \eqref{eq:2.3e}, and \eqref{eq:2.3g} to define the boundary data.

The experiment was ran using $k=1,2$ for the $\mathcal P_k$ finite elements and $\mathcal P_k/\mathcal P_{k-1}$ boundary elements. The time step and mesh size were refined simultaneously and the final time was $t=1.5$.  All errors are measured at the final time: $E^{v}_{h,\kappa}$ measures the maximum error on twenty randomly chosen points in the exterior domain, while elastic and electric fields errors are measured in the $L^2(\Omega_-)$ and $H^1(\Omega_-)$ norms. Tables \ref{tab:3} and \ref{tab:4} along with Figure \ref{fig:3} show the outcome of the convergence tests.
\begin{table}[h]\centering
\scalebox{0.75}{
\begin{tabular}{ccccccccccccc}
\hline
\multicolumn{3}{|c|}{$k=1$} & \multicolumn{4}{|c|}{$L^2(\Omega_-)$} & \multicolumn{4}{|c|}{$H^1(\Omega_-)$} \\
\hline
\multicolumn{1}{|c|}{$h/\kappa$} & \multicolumn{1}{c|}{$E^{v}_{h,\kappa}$} & \multicolumn{1}{c|}{e.c.r.} & \multicolumn{1}{c|}{$E^{\mathbf u}_{h,\kappa}$} & \multicolumn{1}{c|}{e.c.r.} & \multicolumn{1}{c|}{$E^{\psi}_{h,\kappa}$} & \multicolumn{1}{c|}{e.c.r.} & \multicolumn{1}{c|}{$E^{\mathbf u}_{h,\kappa}$} & \multicolumn{1}{c|}{e.c.r.} & \multicolumn{1}{c|}{$E^{\psi}_{h,\kappa}$} & \multicolumn{1}{c|}{e.c.r.} \\ \hline
 2 E-1/7.5 E-2 	&   2.054 E-2 & ---    & 6.363 E-2  & ---   & 4.179 E-2   & ---    & 5.714 E-1 & ---    & 1.702     & ---    \\ \hline
 1 E-1/3.75 E-2 &   7.864 E-3  & 1.385  & 1.726 E-2  & 1.882 & 1.034 E-2 & 2.015  & 2.067 E-1  & 1.467  & 8.515 E-1  & 0.999  \\ \hline
 5 E-2/1.875 E-2 &   1.831 E-3 & 2.102  & 4.537 E-3  & 1.928 & 2.590 E-3  & 1.997  & 8.600 E-2 & 1.265  & 4.258 E-1 & 1.000  \\ \hline
 2.5 E-2/9.375 E-3 &   4.485 E-4  & 2.030  & 1.159 E-3  & 1.969 &  6.485 E-4  & 1.997  & 3.912 E-2 & 1.136  & 2.129 E-1 & 1.000  \\ \hline   
\end{tabular}
}
\caption{{\footnotesize Relative errors and estimated convergence rates in the time domain for the Trapezoidal Rule Convolution Quadrature with $\mathcal P_1$ finite elements and $\mathcal P_1/\mathcal P_0$ boundary elements: $h$ represents the maximum length of the panels used to discretize the boundary, $\kappa$ is the size of the timesteps. The errors are measured at the final time  $T=1.5$.}}\label{tab:3}
\end{table}
\begin{table}[h]\centering
\scalebox{0.75}{
\begin{tabular}{ccccccccccccc}
\hline
\multicolumn{3}{|c|}{$k=2$} & \multicolumn{4}{|c|}{$L^2(\Omega_-)$} & \multicolumn{4}{|c|}{$H^1(\Omega_-)$} \\
\hline
\multicolumn{1}{|c|}{$h/\kappa$} & \multicolumn{1}{c|}{$E^{v}_{h,\kappa}$} & \multicolumn{1}{c|}{e.c.r.} & \multicolumn{1}{c|}{$E^{\mathbf u}_{h,\kappa}$} & \multicolumn{1}{c|}{e.c.r.} & \multicolumn{1}{c|}{$E^{\psi}_{h,\kappa}$} & \multicolumn{1}{c|}{e.c.r.} & \multicolumn{1}{c|}{$E^{\mathbf u}_{h,\kappa}$} & \multicolumn{1}{c|}{e.c.r.} & \multicolumn{1}{c|}{$E^{\psi}_{h,\kappa}$} & \multicolumn{1}{c|}{e.c.r.} \\ \hline
 2 E-1/7.5 E-2  & 3.422 E-2  & ---   & 4.627 E-2  & ---   & 1.544 E-2 & ---   & 6.323 E-1  & ---   & 1.495 E-1     & ---    \\ \hline
 1 E-1/3.75 E-2 & 2.329 E-2  & 0.555 & 1.242 E-2  & 1.898 & 3.722 E-3 & 2.052  & 1.821 E-1  & 1.795 & 3.260 E-2 & 2.197  \\ \hline
 5 E-2/1.875 E-2 & 5.836 E-3  & 1.997 & 3.128 E-3  & 1.989 & 9.194 E-4 & 2.017 & 4.607 E-2  & 1.983 & 7.735 E-3 & 2.076  \\ \hline
 2.5 E-2/9.375 E-3 & 1.444 E-3  & 2.015 & 7.826 E-4  & 1.999 & 2.288 E-4 & 2.007 & 1.151 E-2  & 2.001 & 1.907 E-3 & 2.020  \\ \hline  
\end{tabular}
}
\caption{{\footnotesize Relative errors and estimated convergence rates in the time domain for the Trapezoidal Rule Convolution Quadrature with $\mathcal P_2$ finite elements and $\mathcal P_2/\mathcal P_1$ boundary elements: $h$ represents the maximum length of the panels used to discretize the boundary, $\kappa$ is the size of the timesteps. The errors are measured at the final time  $T=1.5$ .}}\label{tab:4}
\end{table}
\begin{figure}[h]\centering
\includegraphics[height =5cm]{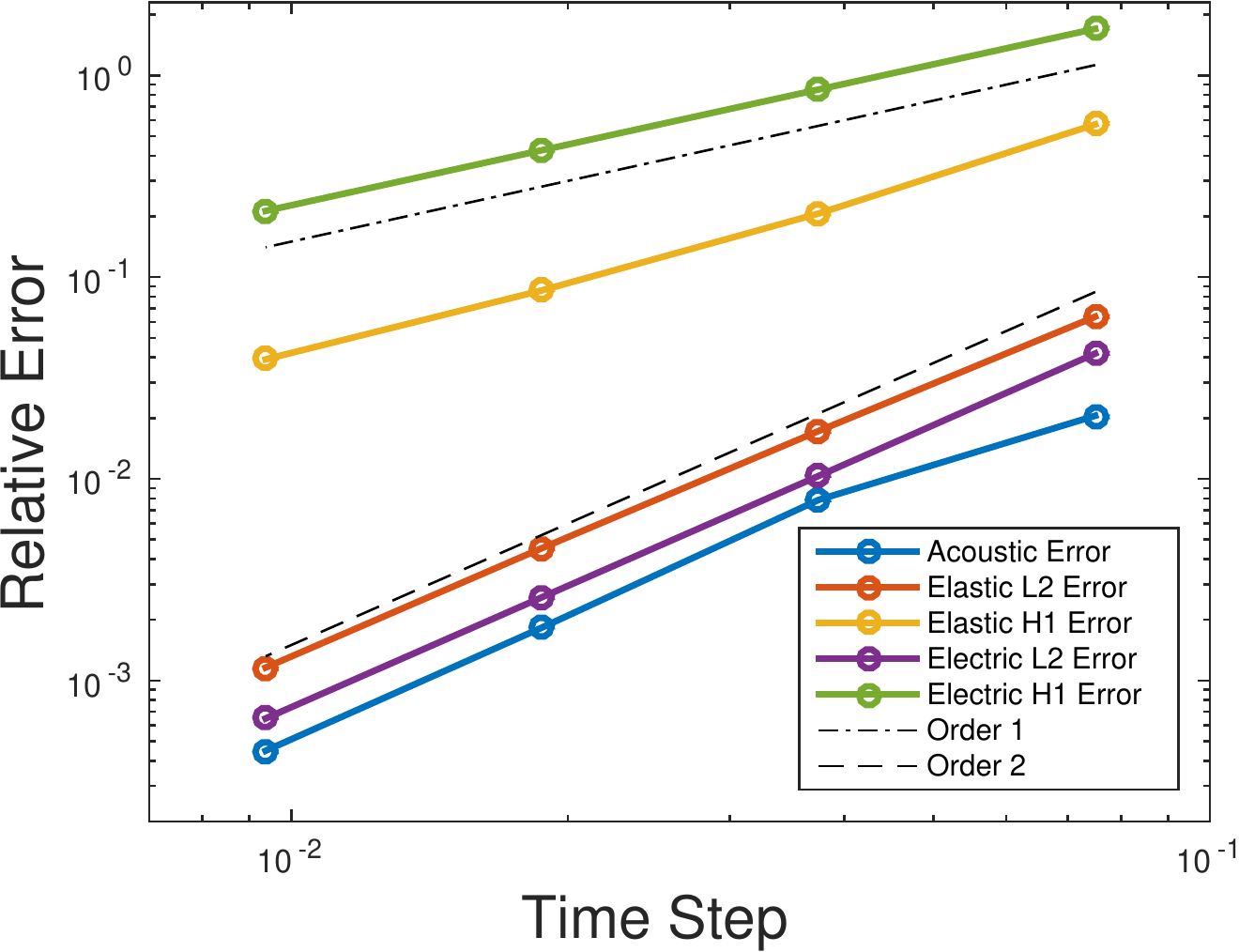}\qquad
\includegraphics[height =5cm]{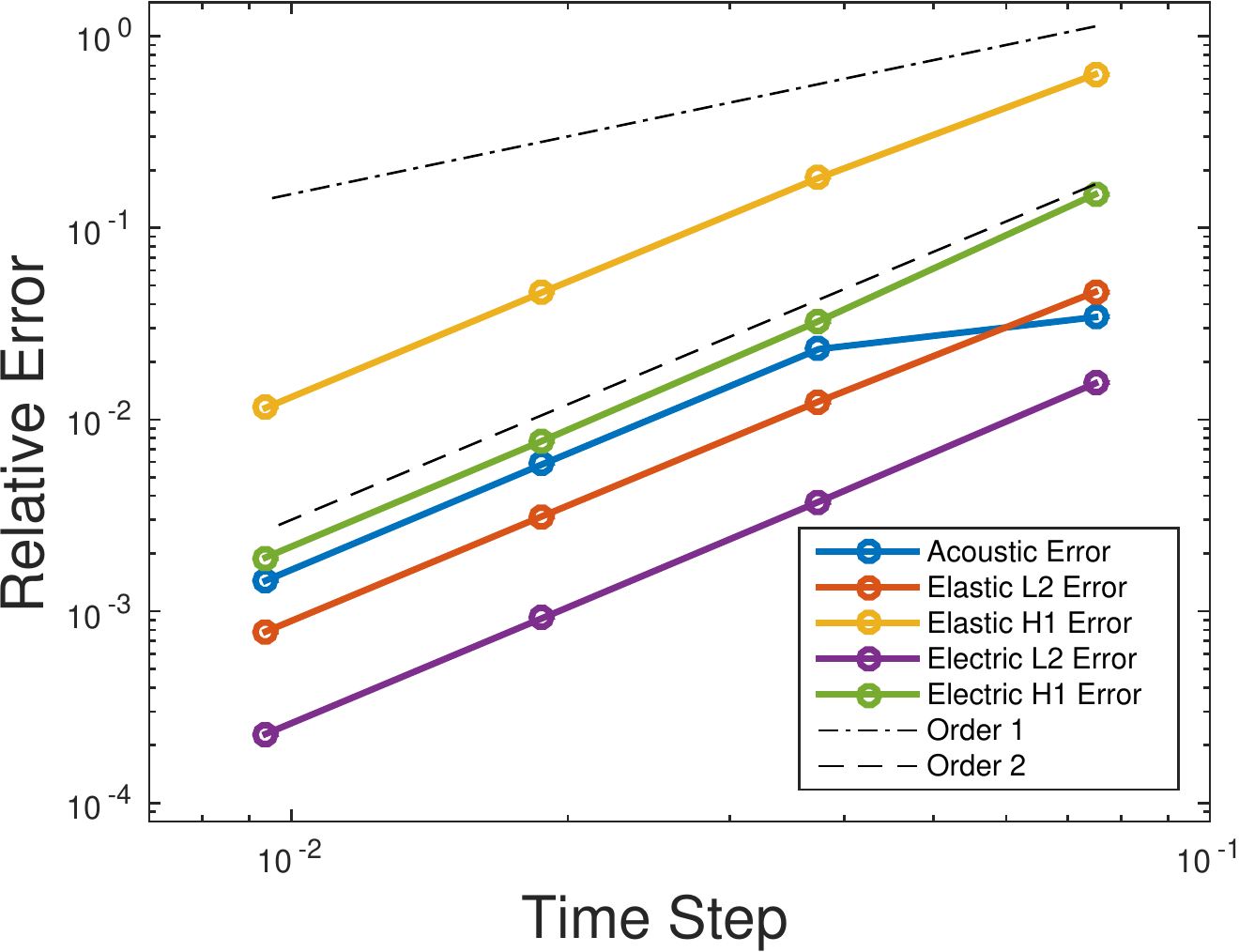}
\caption{{\footnotesize Convergence studies for the Trapezoidal Rule-based time stepping in the case of $\mathcal P_1/\mathcal P_0$ boundary elements and $\mathcal P_1$ finite elements (left) and $\mathcal P_2/\mathcal P_1$ boundary elements and $\mathcal P_2$ finite elements (right).}}\label{fig:3}
\end{figure} 
\begin{table}[h]\centering
\scalebox{0.75}{
\begin{tabular}{ccccccccccccc}
\hline
\multicolumn{3}{|c|}{$k=1$} & \multicolumn{4}{|c|}{$L^2(\Omega_-)$} & \multicolumn{4}{|c|}{$H^1(\Omega_-)$} \\
\hline
\multicolumn{1}{|c|}{$h/\kappa$} & \multicolumn{1}{c|}{$E^{v}_{h,\kappa}$} & \multicolumn{1}{c|}{e.c.r.} & \multicolumn{1}{c|}{$E^{\mathbf u}_{h,\kappa}$} & \multicolumn{1}{c|}{e.c.r.} & \multicolumn{1}{c|}{$E^{\psi}_{h,\kappa}$} & \multicolumn{1}{c|}{e.c.r.} & \multicolumn{1}{c|}{$E^{\mathbf u}_{h,\kappa}$} & \multicolumn{1}{c|}{e.c.r.} & \multicolumn{1}{c|}{$E^{\psi}_{h,\kappa}$} & \multicolumn{1}{c|}{e.c.r.} \\ \hline
 2 E-1/7.5 E-2  	& 2.805 E-2  & ---   & 9.448 E-2  & ---   & 4.772 E-2 & ---   & 7.683 E-1  & ---   & 1.709     & ---    \\ \hline
1 E-1/3.75 E-2 	& 2.543 E-2  & 0.141 & 3.401 E-2  & 1.474 & 1.377 E-2 & 1.793  & 3.931 E-1 & 0.967 & 8.546 E-1 & 1.000  \\ \hline
 5 E-2/1.875 E-2 & 1.571 E-2  & 0.694 & 1.010 E-2 & 1.749 & 3.689 E-3  & 1.900 & 1.513 E-1 & 1.378 & 4.264 E-1 & 1.003  \\ \hline 
2.5 E-2/9.375 E-3 & 4.650 E-3  & 1.757 & 2.655 E-3  & 1.930 & 9.379 E-4 & 1.975 & 5.231 E-2 & 1.532 & 2.130 E-1 & 1.001  \\ \hline 
\end{tabular}
}
\caption{{\footnotesize Relative errors and estimated convergence rates in the time domain for the BDF2-based Convolution Quadrature with $\mathcal P_1$ finite elements and $\mathcal P_1/\mathcal P_0$ boundary elements: $h$ represents the maximum length of the panels used to discretize the boundary, $\kappa$ is the size of the timesteps. The errors are measured at the final time  $T=1.5$ .}}\label{tab:5}
\end{table}
\begin{table}[h]\centering
\scalebox{0.75}{
\begin{tabular}{ccccccccccccc}
\hline
\multicolumn{3}{|c|}{$k=2$} & \multicolumn{4}{|c|}{$L^2(\Omega_-)$} & \multicolumn{4}{|c|}{$H^1(\Omega_-)$} \\
\hline
\multicolumn{1}{|c|}{$h/\kappa$} & \multicolumn{1}{c|}{$E^{v}_{h,\kappa}$} & \multicolumn{1}{c|}{e.c.r.} & \multicolumn{1}{c|}{$E^{\mathbf u}_{h,\kappa}$} & \multicolumn{1}{c|}{e.c.r.} & \multicolumn{1}{c|}{$E^{\psi}_{h,\kappa}$} & \multicolumn{1}{c|}{e.c.r.} & \multicolumn{1}{c|}{$E^{\mathbf u}_{h,\kappa}$} & \multicolumn{1}{c|}{e.c.r.} & \multicolumn{1}{c|}{$E^{\psi}_{h,\kappa}$} & \multicolumn{1}{c|}{e.c.r.} \\ \hline
 2 E-1/7.5 E-2  & 2.959 E-2  & ---   & 9.368 E-2  & ---   & 2.999 E-2  & ---   & 8.178 E-1 & ---   & 2.287 E-1     & ---    \\ \hline
1 E-1/3.75 E-2 	& 3.047 E-2  & -0.041 & 3.884 E-2  & 1.270 & 1.247 E-2  & 1.265  & 4.699 E-1 & 0.799 & 1.097 E-1 & 1.059  \\ \hline
 5 E-2/1.875 E-2 & 1.958 E-2  & 0.638 & 1.186 E-2  & 1.712 & 3.566 E-3 & 1.806 & 1.664 E-1  & 1.498 & 3.084 E-3 & 1.832  \\ \hline 
2.5 E-2/9.375 E-3 & 5.680 E-3  & 1.785 & 3.099 E-3  & 1.936 & 9.102 E-4 & 1.970 & 4.511 E-2 & 1.883 & 7.617 E-3 & 2.018  \\ \hline 
\end{tabular}
}
\caption{{\footnotesize Relative errors and estimated convergence rates in the time domain for the BDF2-based Convolution Quadrature with $\mathcal P_2$ finite elements and $\mathcal P_2/\mathcal P_1$ boundary elements: $h$ represents the maximum length of the panels used to discretize the boundary, $\kappa$ is the size of the timesteps. The errors are measured at the final time  $T=1.5$ .}}\label{tab:6}
\end{table}
\begin{figure}[h]\centering
\includegraphics[height =5cm]{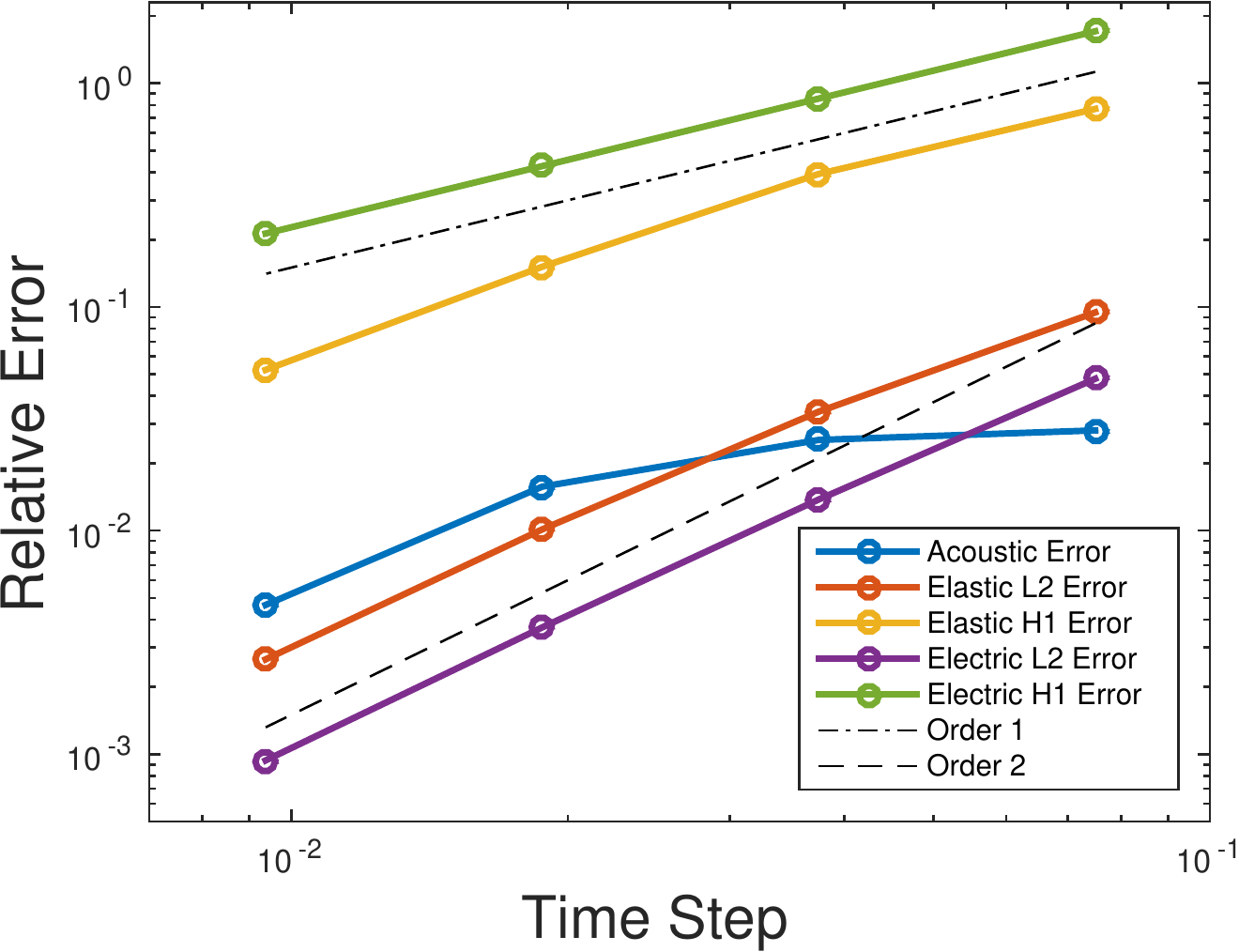}\qquad
\includegraphics[height =5cm]{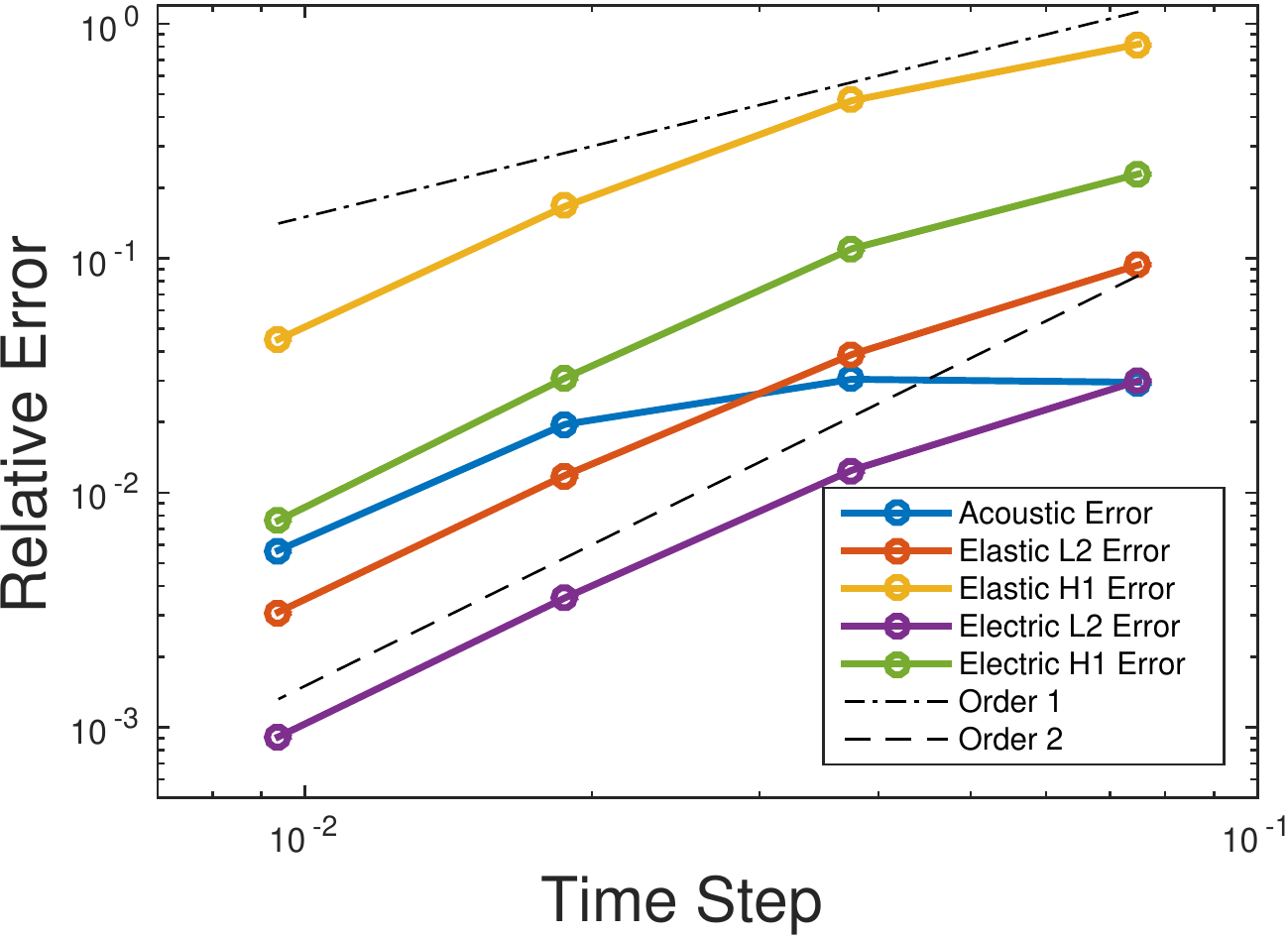}
\caption{{\footnotesize Convergence studies for the BDF2-based time stepping in the case of $\mathcal P_1/\mathcal P_0$ boundary elements and $\mathcal P_1$ finite elements (left) and $\mathcal P_2/\mathcal P_1$ boundary elements and $\mathcal P_2$ finite elements (right).}}\label{fig:4}
\end{figure} 
%
\paragraph{A sample simulation.}
%
As an example, we consider the interaction between the acoustic pulse
\[
v^{inc}= 3\chi_{[0,0.3]}(88s)\sin{(88s)},\quad s:= (t-\mathbf r\cdot\mathbf d),\quad \mathbf r:=(x,y),\quad \mathbf d := (1,5)/\sqrt{26},
\]
and a pentagonal piezoelectric scatterer with mass density given by
\[
\rho = 5 + 25e^{-(10|\mathbf r|)^2}.
\]
The remaining physical parameters of the solid were taken to be those defined by \eqref{eq:00} and the entire solid/fluid interface was taken as Dirichlet boundary, where a grounding potential $\psi\equiv 0$ was imposed as boundary condition for all times. The simulation was carried out using $\mathcal P_2$ Lagrangian finite elements and $\mathcal P_2/\mathcal P_1$ continuous/discontinuous Galerkin boundary elements with Trapezoidal Rule-based time discretization using a time step $\kappa=0.005$. Figures \ref{fig:5} to \ref{fig:7} show snapshots of the process at different times.
\begin{figure}[h]\centering
\includegraphics[height =7cm]{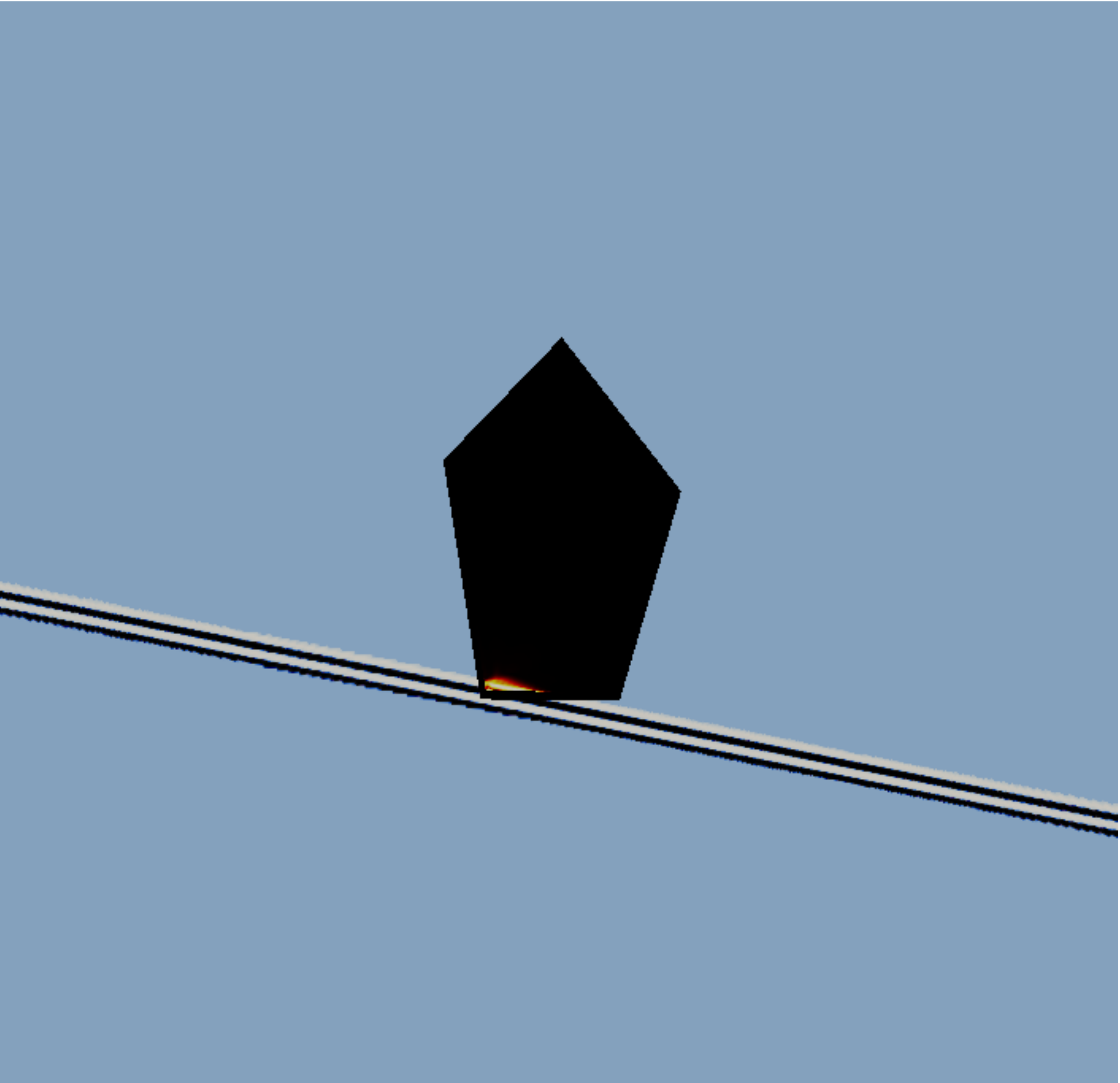}
\includegraphics[height =7cm]{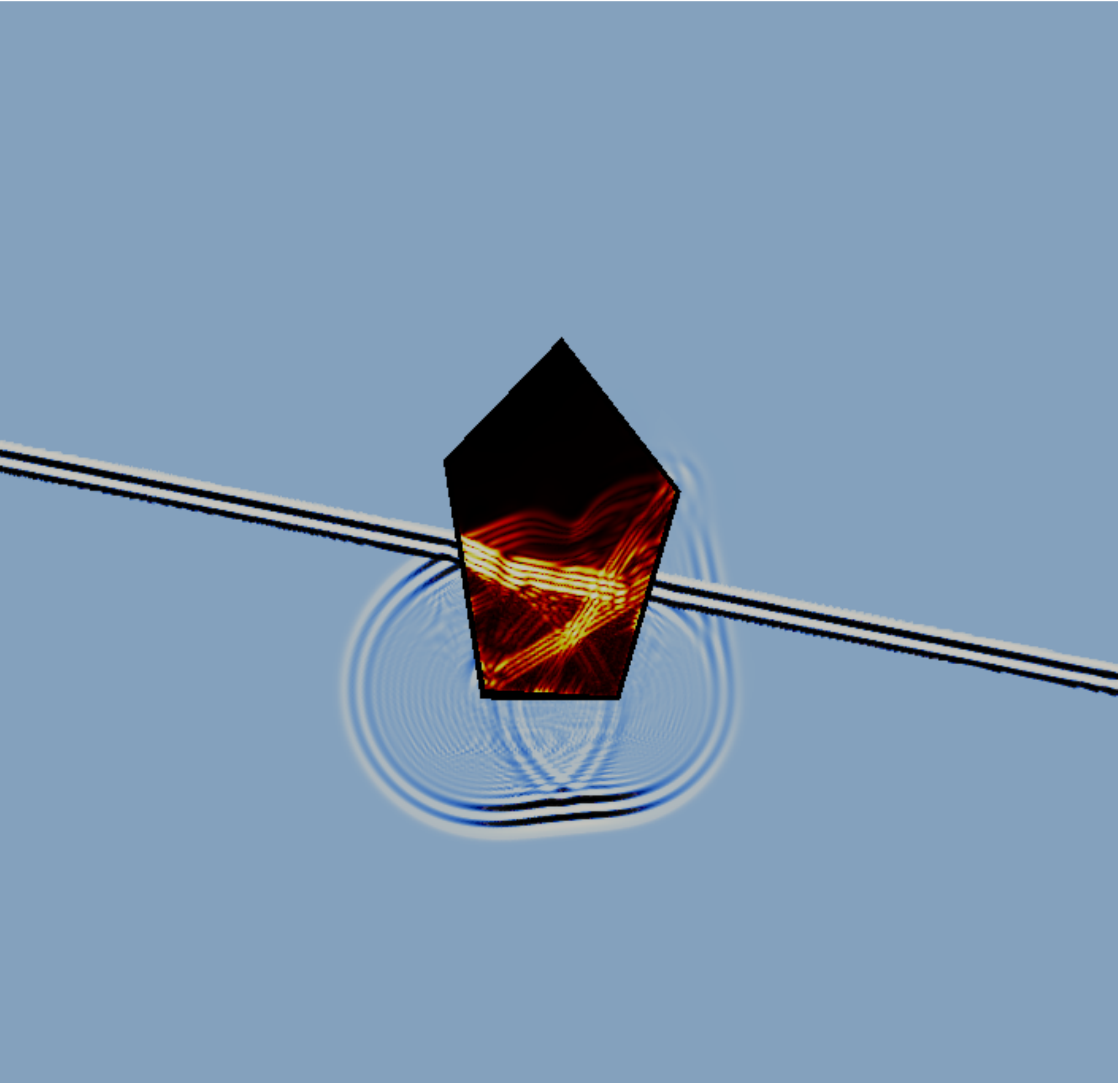}
\includegraphics[height =7cm]{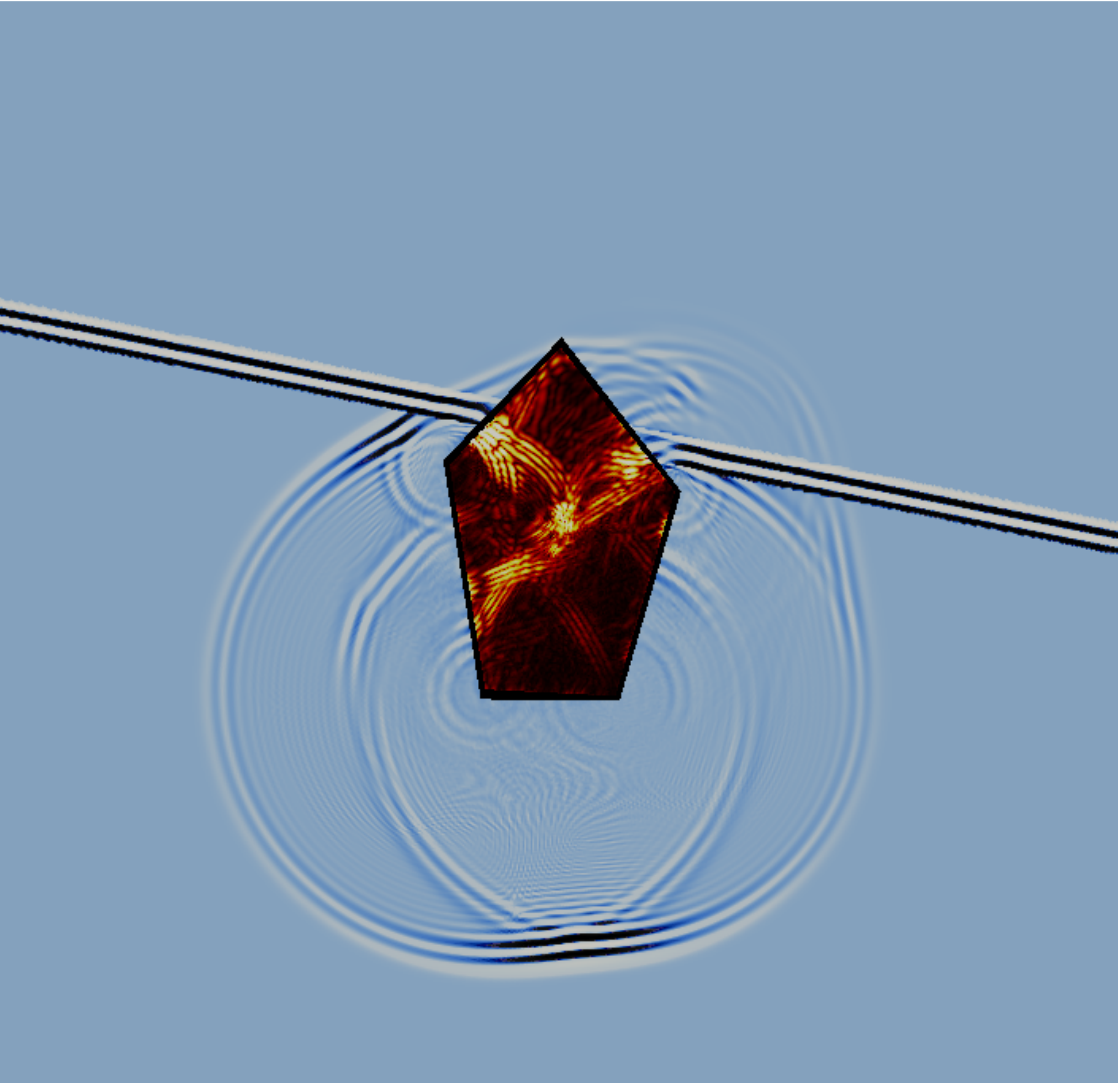}
\includegraphics[height =7cm]{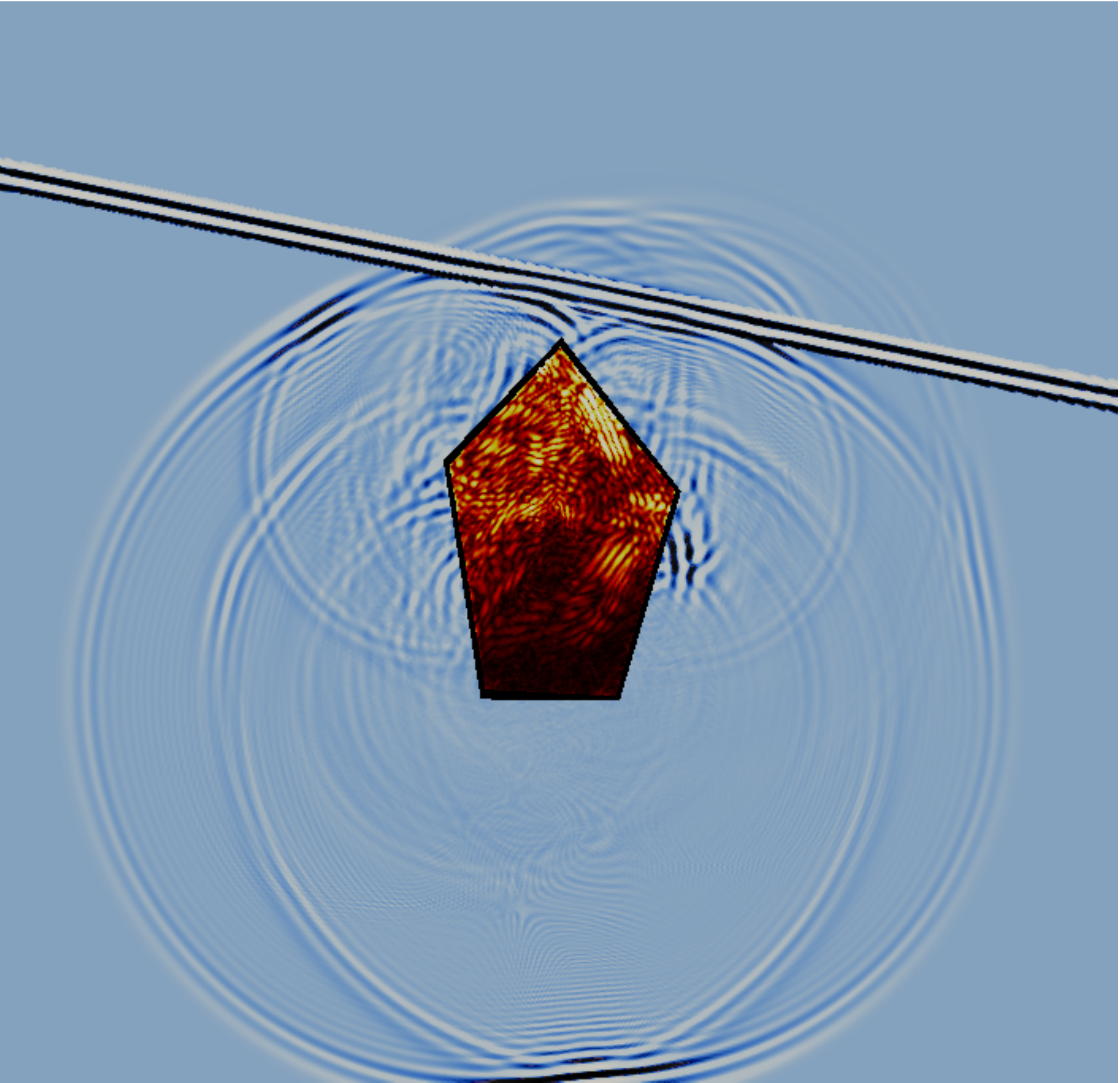}
\caption{{\footnotesize The total acoustic wave shown at times $t=0.175, 0.7, 1.225, 1.75$.}}\label{fig:5}
\end{figure} 

\begin{figure}[h]\centering
\includegraphics[height =3.25cm]{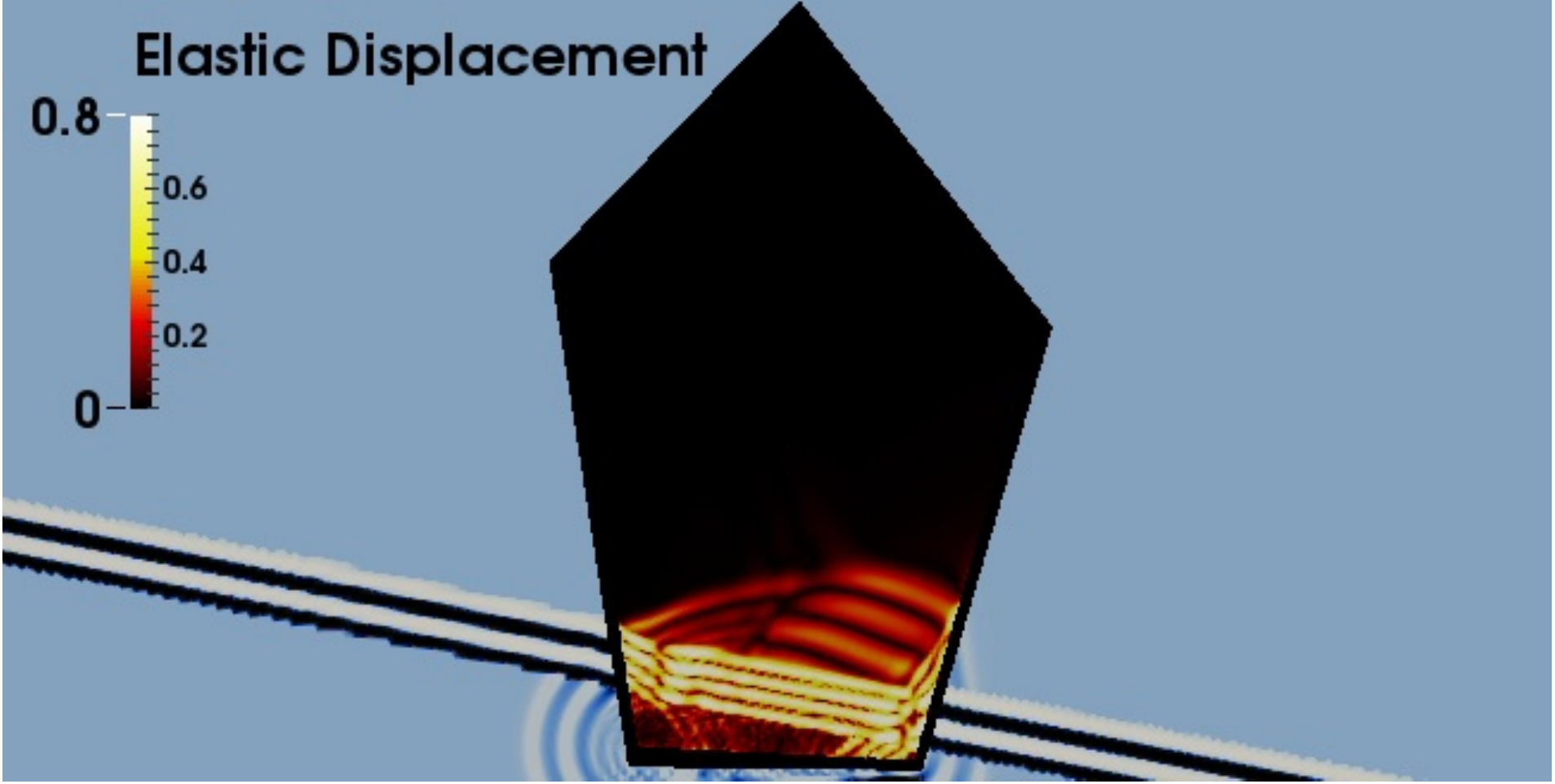}
\includegraphics[height =3.25cm]{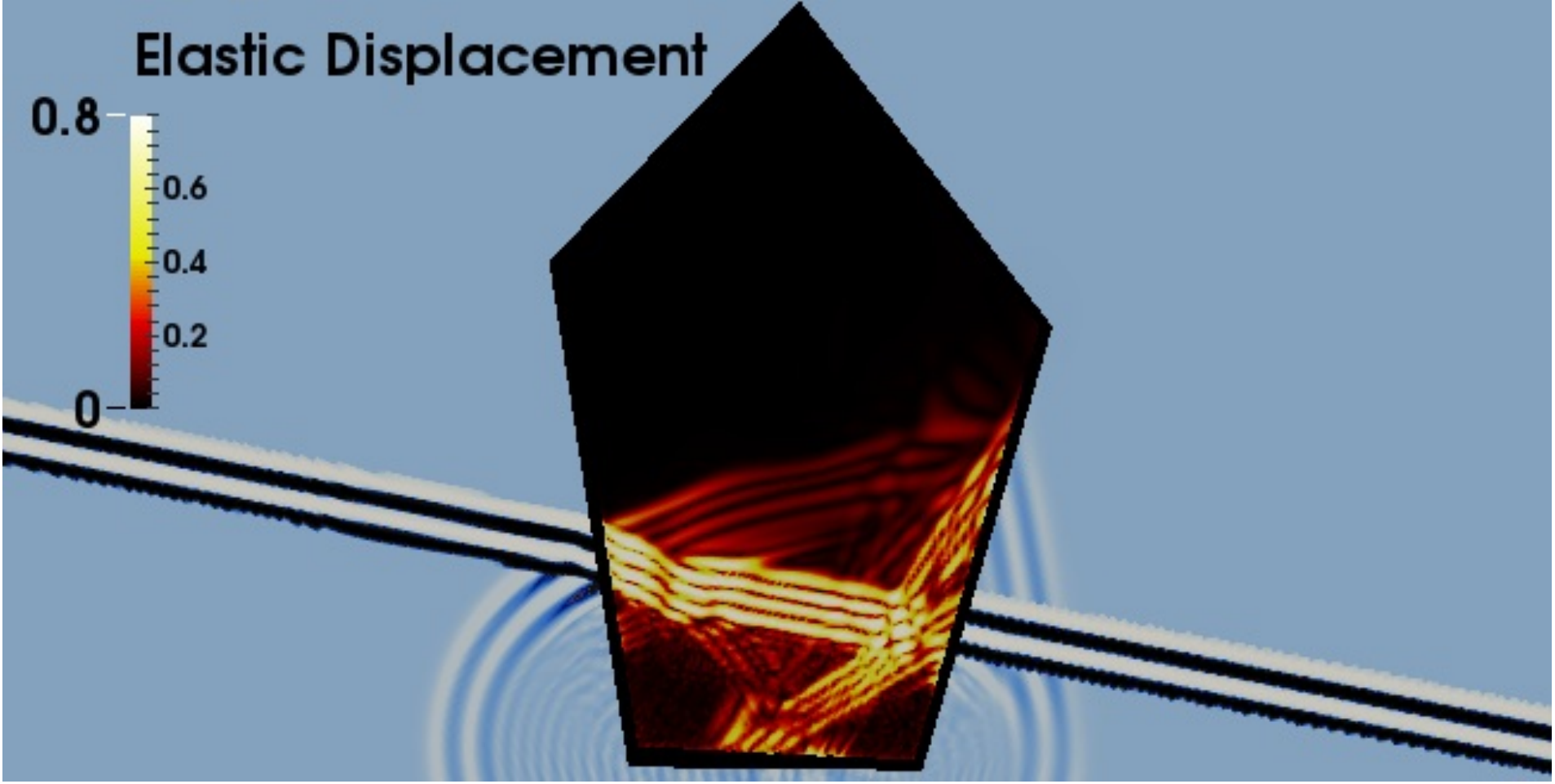}
\includegraphics[height =3.25cm]{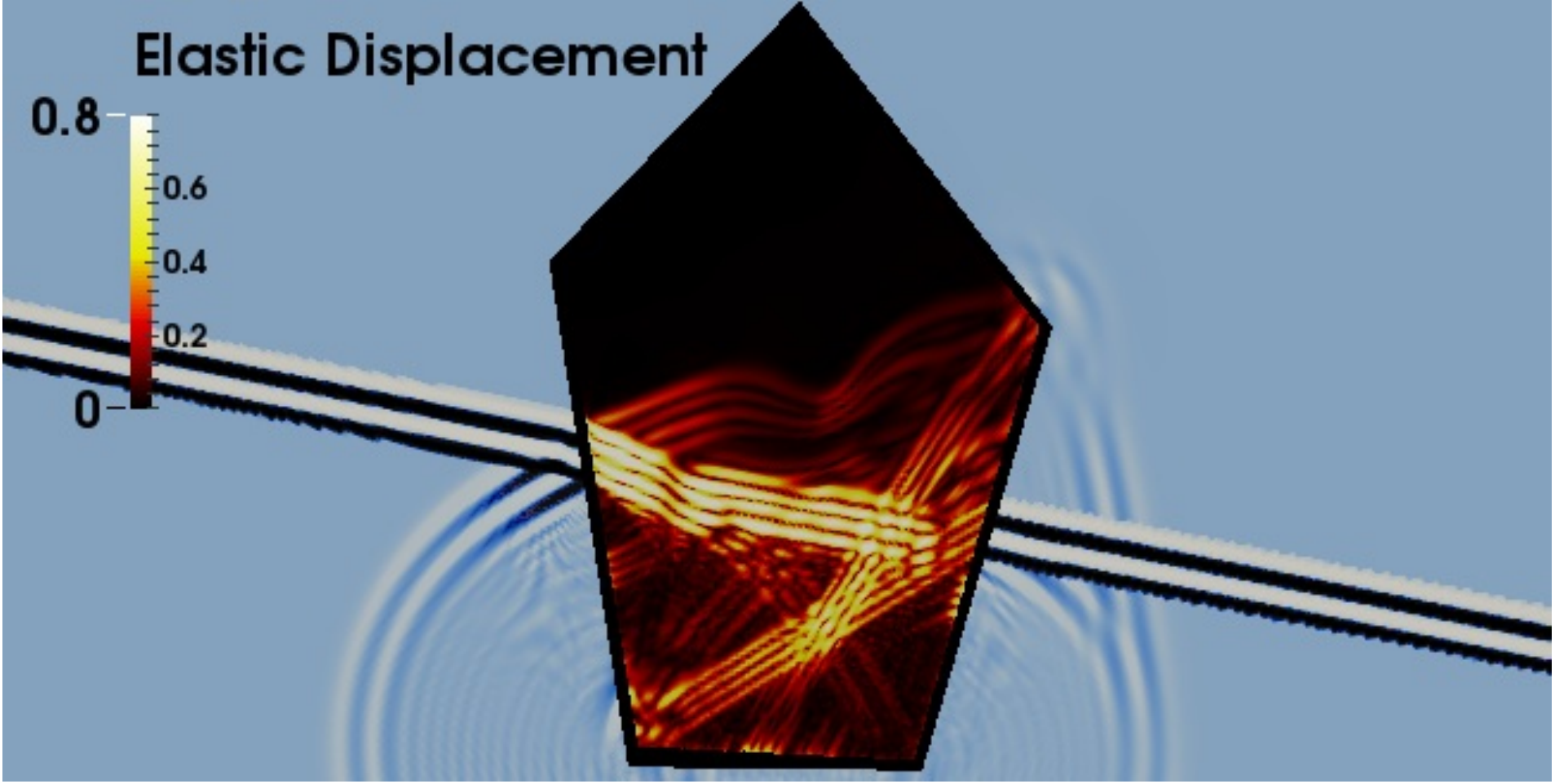}
\includegraphics[height =3.25cm]{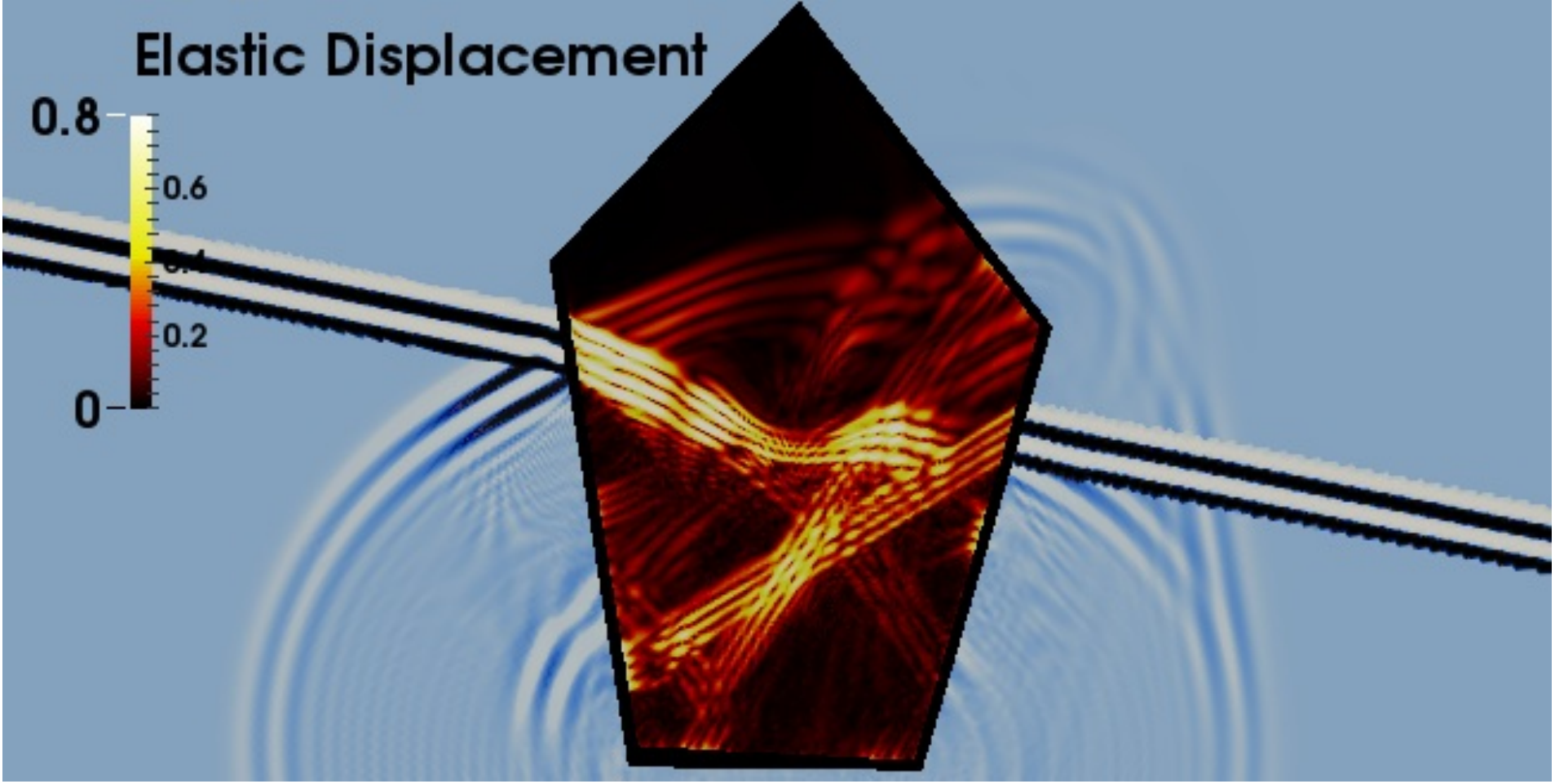}
\includegraphics[height =3.25cm]{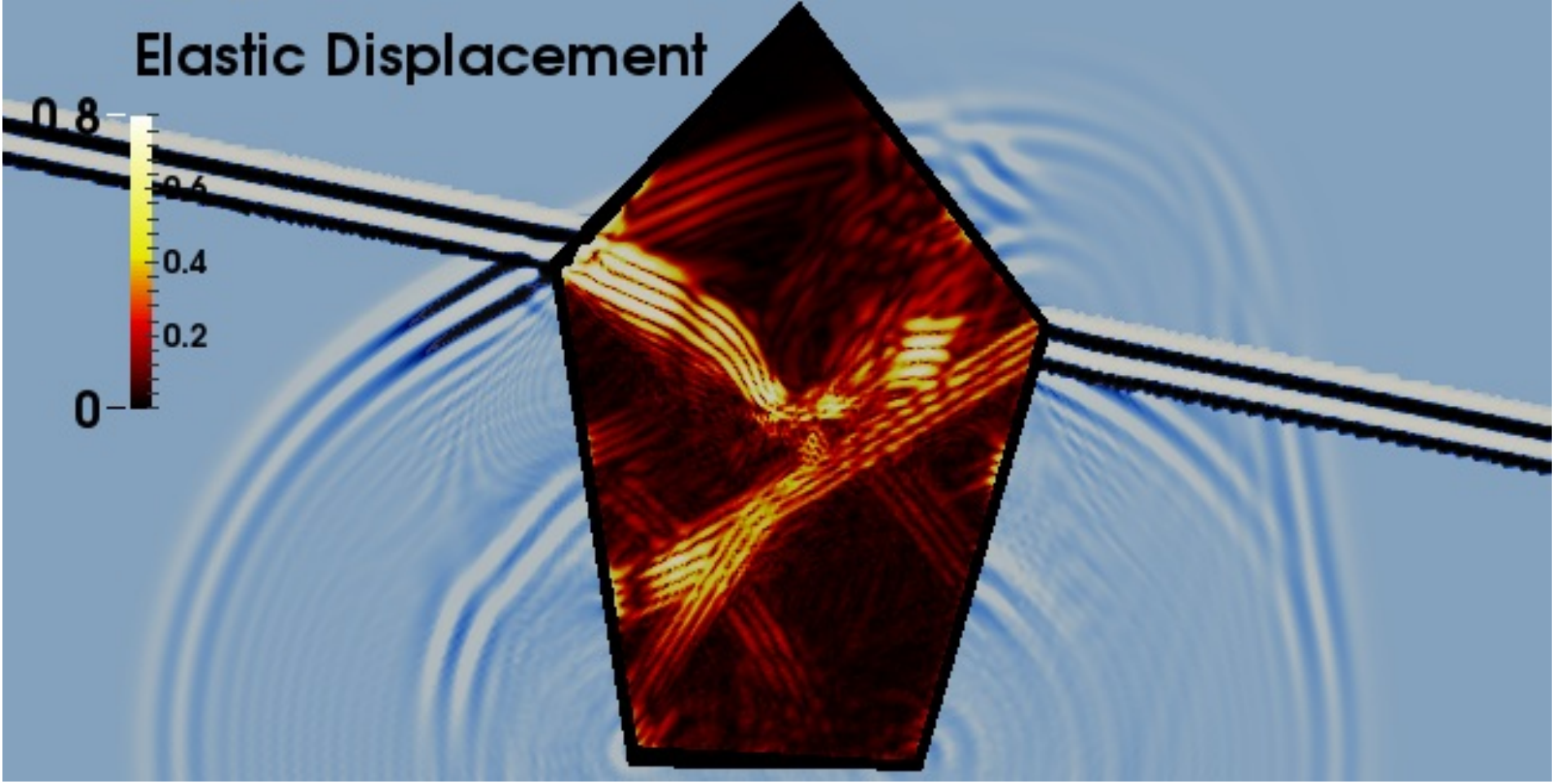}
\includegraphics[height =3.25cm]{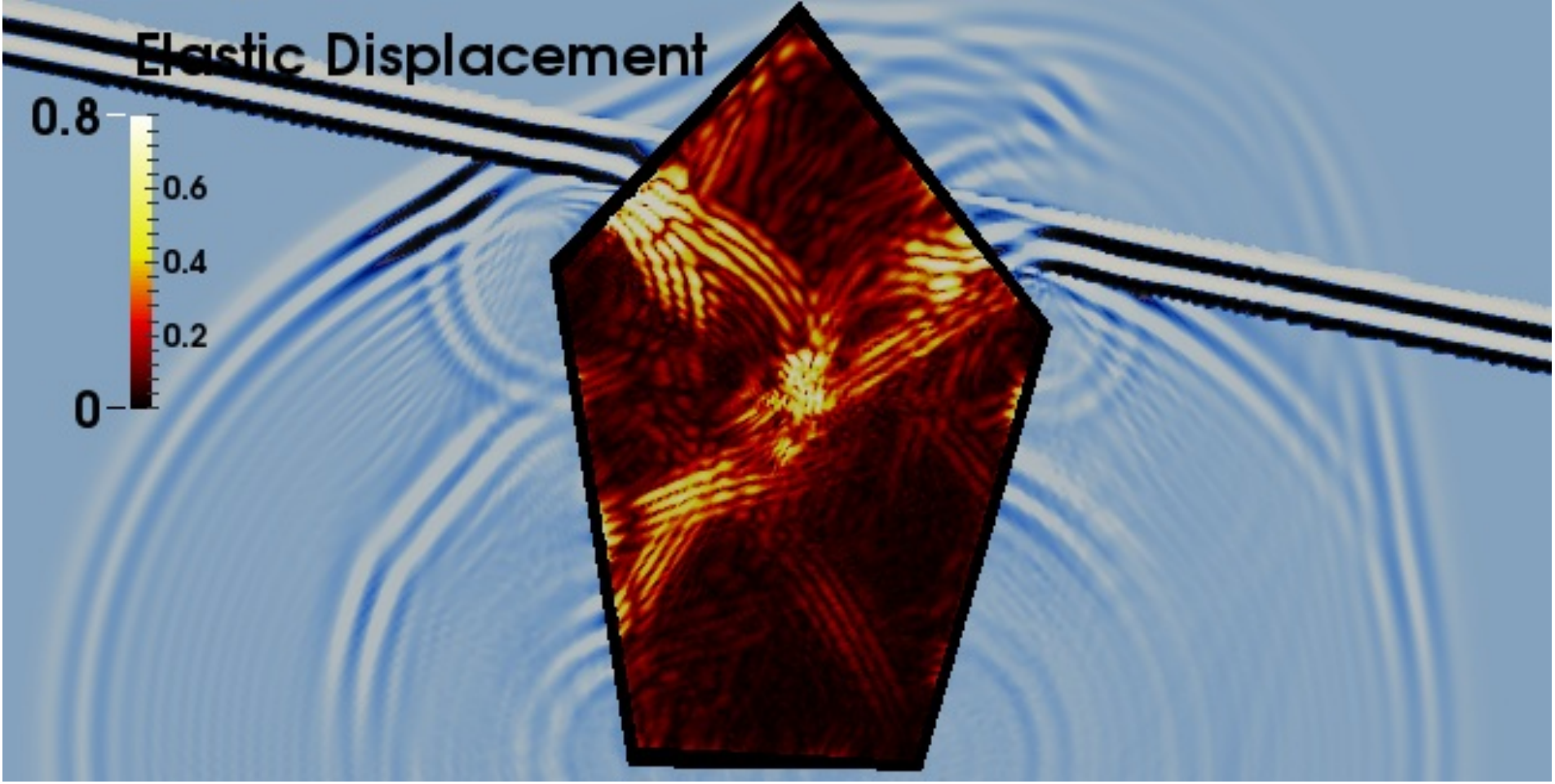}
\caption{{\footnotesize Close up of the norm of the elastic displacement at times $t=0.35, 0.525, 0.7, 0.875, 1.05, 1.225 $.}}\label{fig:6}
\end{figure} 

\begin{figure}[h]\centering
\includegraphics[height =3.25cm]{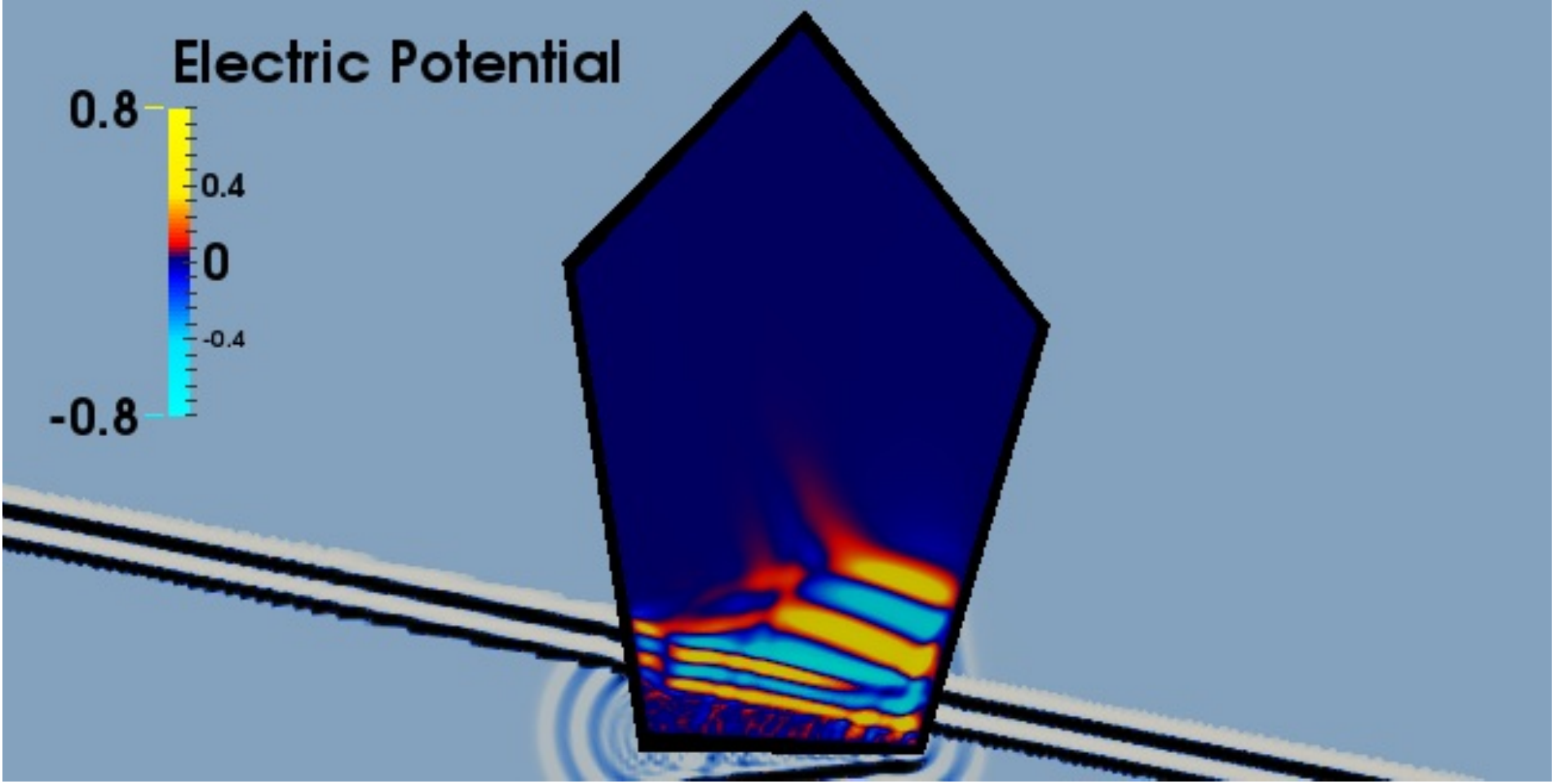}
\includegraphics[height =3.25cm]{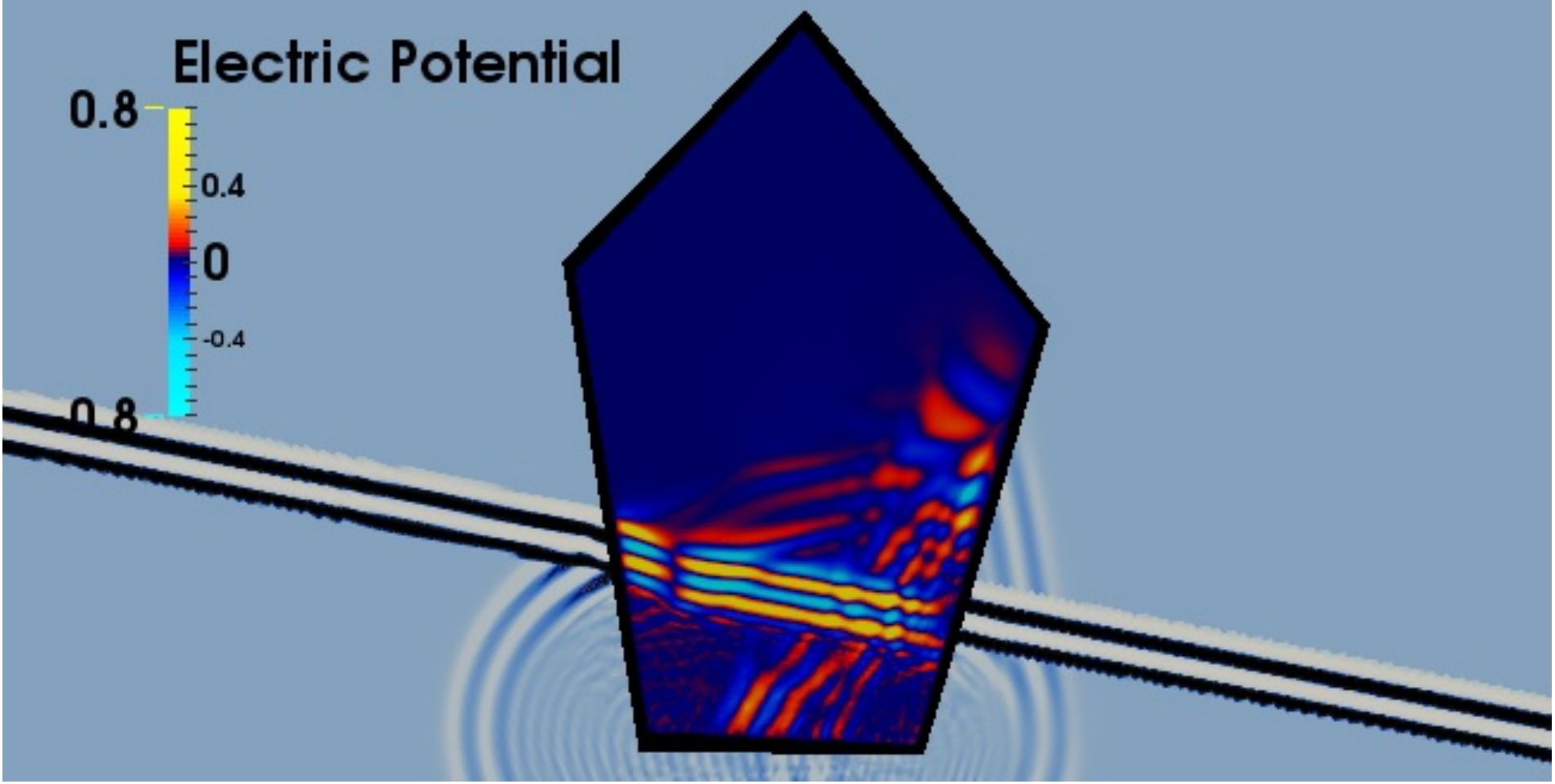}
\includegraphics[height =3.25cm]{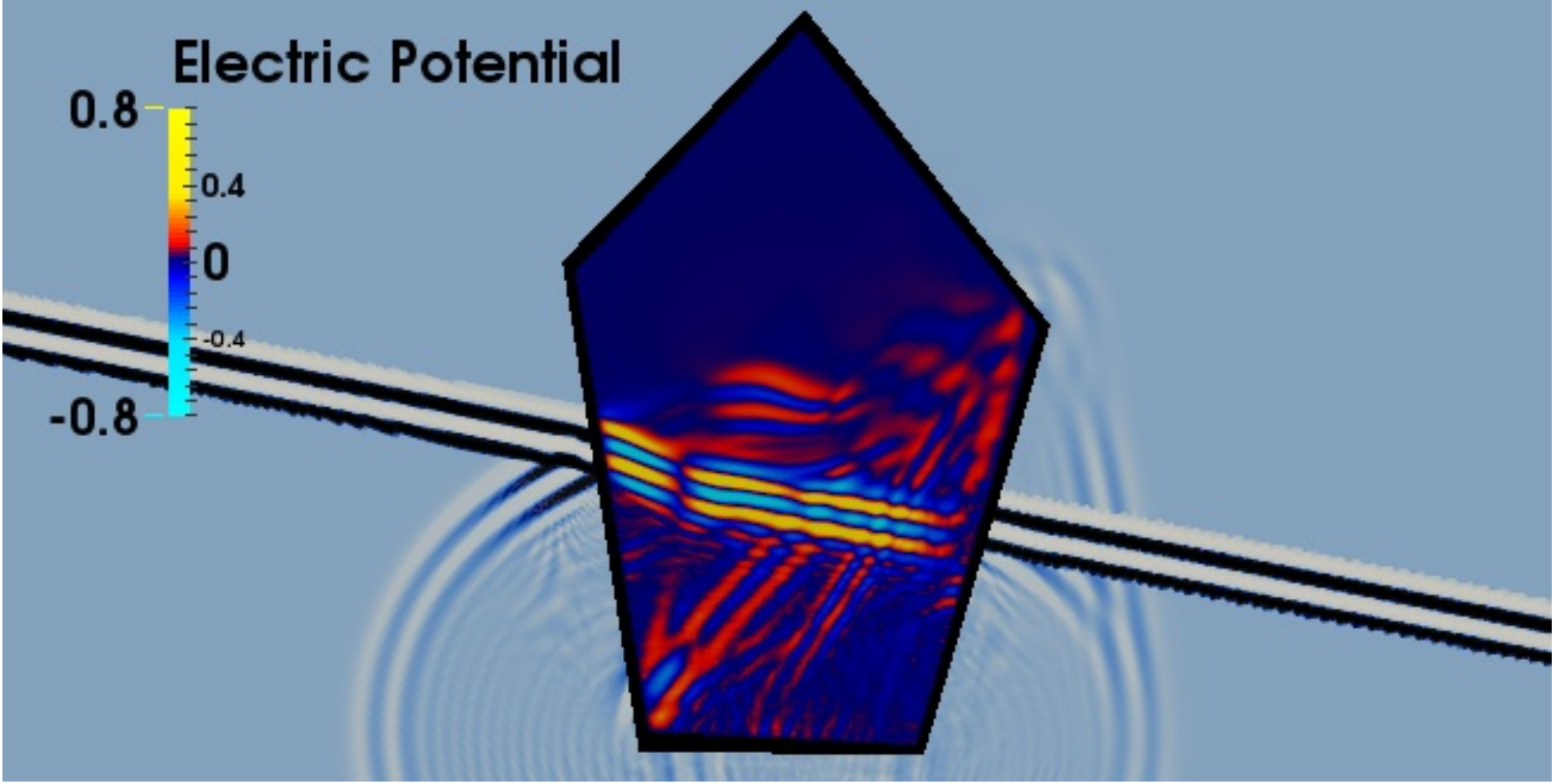}
\includegraphics[height =3.25cm]{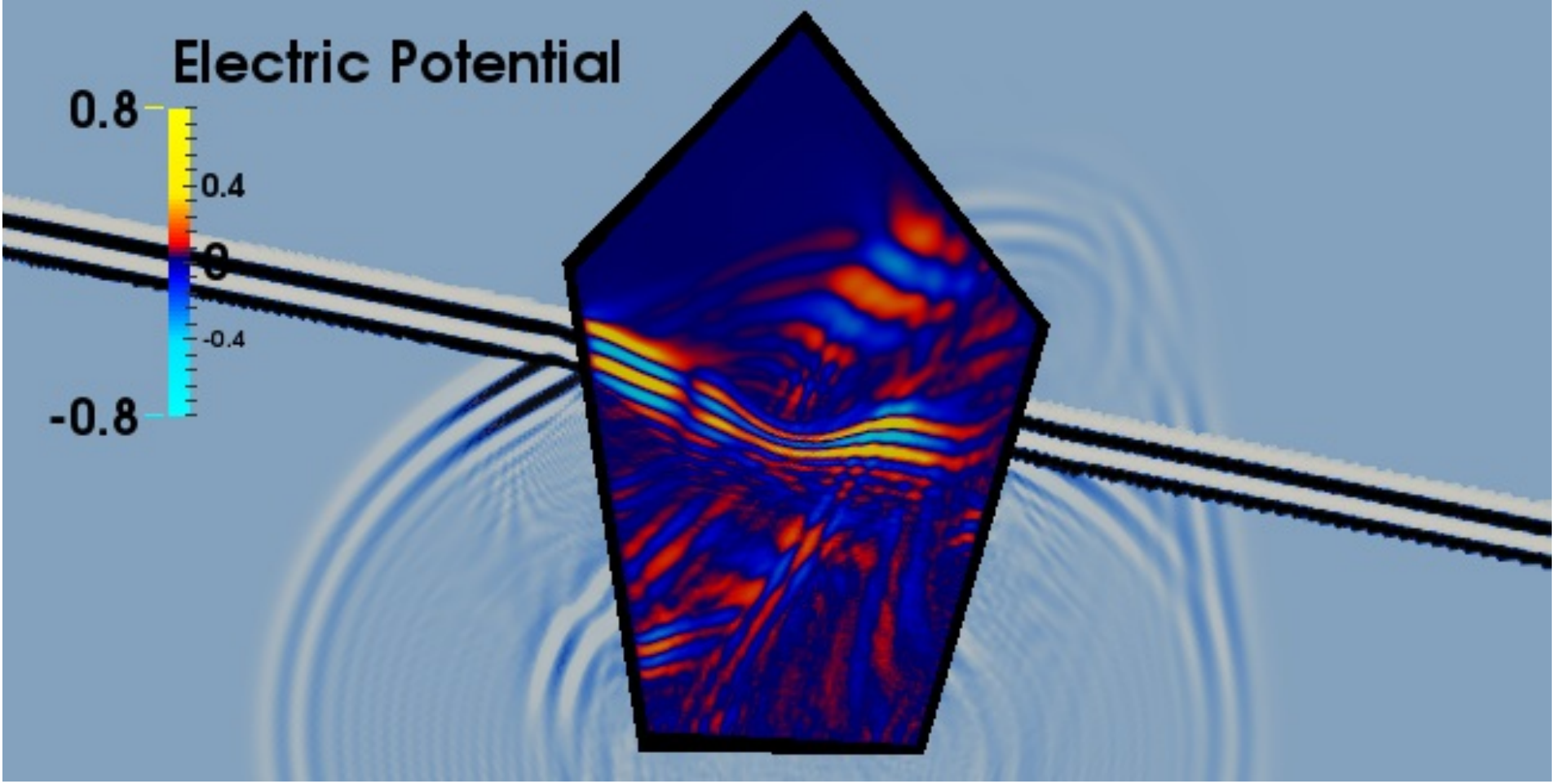}
\includegraphics[height =3.25cm]{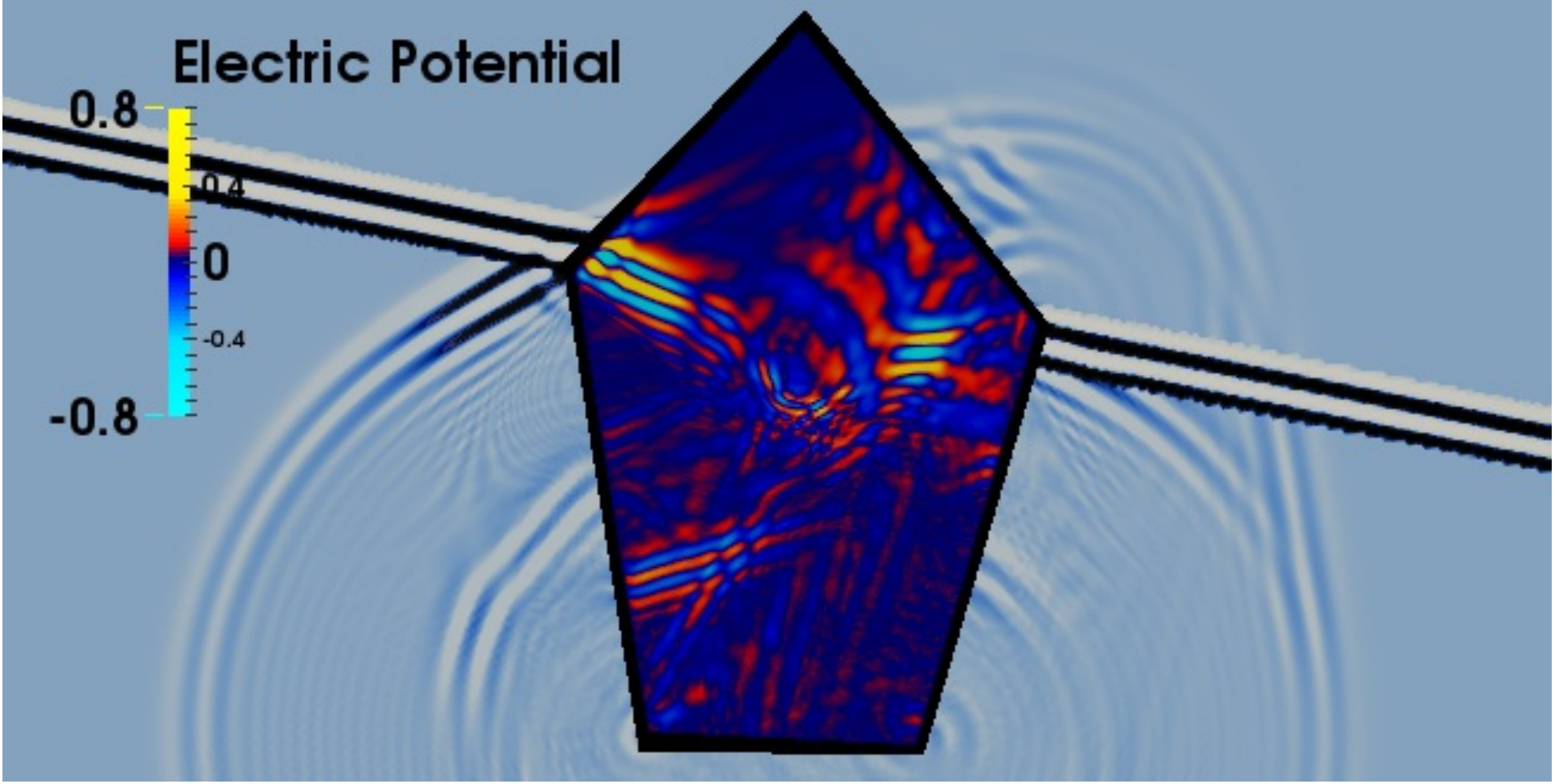}
\includegraphics[height =3.25cm]{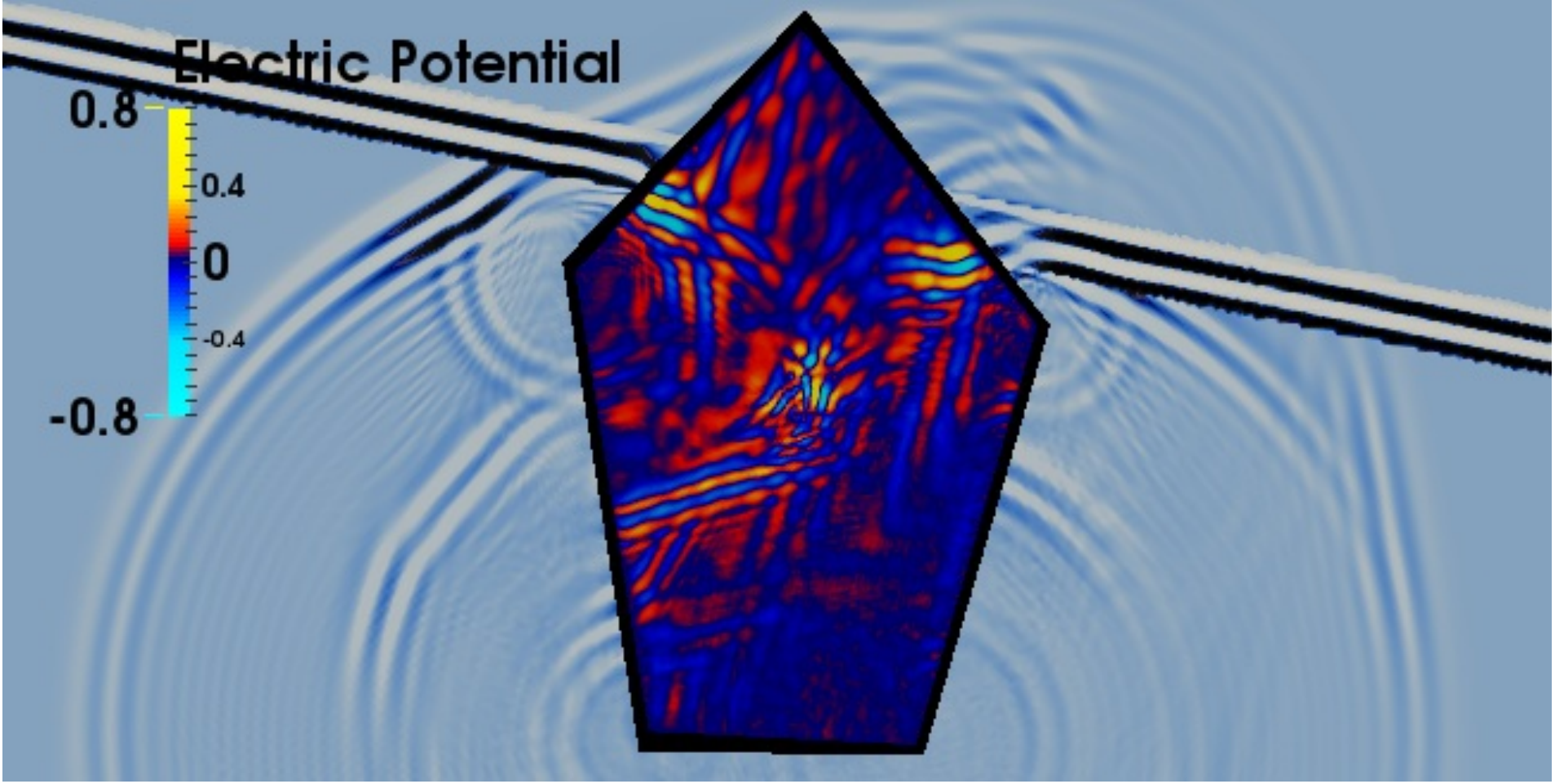}
\caption{{\footnotesize Close up of the electric potential at times $t=0.35, 0.525, 0.7, 0.875, 1.05, 1.225 $.}}\label{fig:7}
\end{figure} 
%
\section{Conclusions}
%
We have developed a well-posed integro-differential formulation for the interaction of acoustic waves and piezoelectric scatterers with Lipschitz boundaries in the transient regime. The formulation is geared towards a numerical implementation employing boundary elements for the acoustic field in the unbounded exterior domain and finite elements for the elastic and electric fields inside of the bounded scatterer, and can be easily implemented computationally building on existing FEM and BEM routines.

We have shown that the resulting stability bounds in the Laplace domain can be used to give explicit time-domain error bounds when BDF2-CQ is used for time discretization, resulting in a quasi-optimal in time scheme of order 2 for sufficiently smooth problem data. Numerical evidence strongly suggests that the Trapezoidal Rule based method has the same optimal convergence order and in fact may have better numerical properties such as a much shorter pre-asymptotic regime and smaller error constants than those for BDF2. 

\clearpage
%
\bibliography{referencesBEM}
%
\end{document}